\documentclass[aap,MSNbibl,dvips]{arximspdf}

\doi{10.1214/13-AAP961} 
\volume{24}
\issue{4}
\pubyear{2014}
\firstpage{1739}
\lastpage{1766}

\makeatletter
\newcommand{\widebar}{\overline}
\newcommand{\rrvert}{\vert}
\newcommand{\llvert}{\vert}
\newtheorem{theorem}{Theorem}
\newtheorem{corollary}{Corollary}
\newproclaim{definition}{Definition}
\newproclaim{example}{Example}
\newtheorem{lemma}{Lemma}
\newtheorem{problem}{Problem}
\newtheorem{proposition}{Proposition}
\newproclaim{remark}{Remark}
\makeatother

\begin{document}
\begin{frontmatter}

\title{Filtration shrinkage, strict local\vspace*{2pt} martingales and~the F\"ollmer measure}
\runtitle{Filtration shrinkage}
\pdftitle{Filtration shrinkage, strict local martingales and the Follmer measure}

\begin{aug}
\author{\fnms{Martin} \snm{Larsson}\corref{}\ead[label=e1]{martin.larsson@epfl.ch}}
\runauthor{M. Larsson}
\affiliation{Swiss Finance Institute, \'Ecole Polytechnique F\'ed\'erale de Lausanne}
\address{Swiss Finance Institute, EPFL\\
Quartier UNIL-Dorigny, Extranef 244\\
1015 Lausanne\\
Switzerland\\
\printead{e1}} 
\end{aug}

\received{\smonth{12} \syear{2012}}
\revised{\smonth{7} \syear{2013}}

%
\begin{abstract}
When a strict local martingale is projected onto a subfiltration to
which it is not adapted,
the local martingale property may be lost, and the finite variation
part of the projection may
have singular paths. This phenomenon has consequences for arbitrage
theory in mathematical finance.
In this paper it is shown that the loss of the local martingale
property is related to a measure
extension problem for the associated F\"ollmer measure. When a solution
exists, the finite variation
part of the projection can be interpreted as the compensator, under the
extended measure, of the
explosion time of the original local martingale. In a topological
setting, this leads to intuitive
conditions under which its paths are singular. The measure extension
problem is then solved
in a Brownian framework, allowing an explicit treatment of several
interesting examples.
\end{abstract}

%
\begin{keyword}[class=AMS]
\kwd[Primary ]{60G44}
\kwd[; secondary ]{60G17}
\kwd{60G30}
\kwd{60G40}
\end{keyword}
\begin{keyword}
\kwd{Filtration shrinkage}
\kwd{local martingales}
\kwd{F\"ollmer measure}
\end{keyword}
\pdfkeywords{60G44, 60G17, 60G30, 60G40, Filtration shrinkage, local martingales, Follmer measure}

\end{frontmatter}

\section{Introduction} \label{SB}

It is a simple fact that the optional projection of a martingale onto a
subfiltration is again a martingale. However, for local martingales the
situation is different, and this was the starting point for F\"ollmer
and Protter in~\cite{FollmerProtter2011}. They consider, among other
things, three-dimensional Brownian motion $B=(B^1,B^2,B^3)$ starting
from $(1,0,0)$, defined on a filtered probability space $(\Omega,\mathcal G,\mathbb G,P)$ where the filtration $\mathbb G=(\mathcal
G_t)_{t\ge0}$ is generated by $B$. In this setting they study optional
projections of the process $N=1/\|B\|$ onto subfiltrations $\mathbb
F^1=(\mathcal F^1_t)_{t\ge0}$ and $\mathbb F^{1,2}=(\mathcal
F^{1,2}_t)_{t\ge0}$ generated by $B^1$ and $(B^1,B^2)$, respectively.
It is well known that $N$, the reciprocal of a BES(3) process, is a
local martingale in~$\mathbb G$. The same turns out to be true for its
optional projection onto $\mathbb F^{1,2}$. However, the optional
projection onto $\mathbb F^1$ \emph{is not a local martingale}. Indeed,
it was shown in~\cite{FollmerProtter2011}, Theorem~5.1, that the equality
%
%
\begin{equation}
\label{eqFP} E^P \bigl[N_t\mid\mathcal
F^1_t \bigr] = 1 + \int_0^t
u_x\bigl(s,B^1_s\bigr)\,dB^1_s
- \int_0^t \frac{1}{s}\,dL^0_s,
\qquad t\ge0,
\end{equation}
holds $P$-a.s., where the function $u$ is given by
\[
u(t,x)= \sqrt{\frac{2\pi}{t}} \exp \biggl( \frac{x^2}{2t} \biggr) \bigl( 1
- \Phi\bigl(|x|/\sqrt{t}\bigr) \bigr)
\]
and $L^0$ is the local time of $B^1$ at level zero. Here $\Phi(\cdot)$
is the standard Normal cumulative distribution function. A superficial
reason for the appearance of the local time is the
nondifferentiability of $u$ at $x=0$, but this is of course highly
specific to this particular example. The main goal of the present paper
is to shed further light on when the optional projection of a general
positive local martingale $N$ fails to be a local martingale, and, when
this is the case, what can be said about the behavior of its finite
variation part. The basic structural result holds for arbitrary
positive local martingales, subject only to a weak regularity condition
on the filtration.

A crucial tool in the analysis is a variant of the \emph{F\"ollmer
measure} $Q_0$ associated with~$N$, whose construction we briefly
review in Section~\ref{SFM}. A~nonuniqueness property of (this
variant of) the F\"ollmer measure leads us to formulate an \emph
{equivalent measure extension problem} (Problem~\ref{MEP}): find an
extension $Q$ of $Q_0$ that is equivalent to $P$ on each $\sigma$-field
of the subfiltration under consideration. When a~solution exists, one
can interpret the finite variation part of the projection of $N$ as the
compensator of a certain stopping time (Theorem~\ref{Tcomp}). This
stopping time is the explosion time of $N$, which may be finite under
the F\"ollmer measure. These developments, valid in full generality,
are carried out in Section~\ref{SMEP}. We then proceed in
Section~\ref{SY} to study filtrations generated by the image under
some continuous
map of the coordinate process $Y$ (we now restrict ourselves to path
space), and take $N$ to be a deterministic function of $Y$. This
additional structure makes it possible to obtain more detailed results
about the points of increase of the finite variation part of the
projection of $N$ (Theorem~\ref{TdL}). As a consequence
(Corollary~\ref{CJY}) we obtain a~simple sufficient condition for its paths to be
singular. Next, in Section~\ref{SMEPsol}, we address the problem of
actually finding a solution to the equivalent measure extension
problem. The setting is now restricted further: the coordinate process
is assumed to be (multidimensional) Brownian motion under~$P$. In this
framework we derive explicit conditions under which the equivalent
measure extension problem can be solved (Theorem~\ref{Tmain11}).
Several illustrating examples are given in Section~\ref{Sex},
including the aforementioned example of F\"ollmer and Protter.

Strict local martingales are fundamental in financial models for asset
pricing bubbles and relative arbitrage; see, for instance, \cite
{HestonLoewensteinWillard2007,JarrowProtterShimbo2010,Fernholz2010fk,Hugonnier2012,Larsson2013}. They also appear in the
so-called Benchmark approach~\cite{PlatenHeath2009}. The role of
filtration shrinkage in this context, in particular the loss of the
local martingale property, is discussed in~\cite{FollmerProtter2011}.
The authors explain how less informed investors may perceive arbitrage
opportunities where there are none; see also~\cite
{JarrowProtter2013}. Applications in credit risk include~\cite
{Cetinetal2004} and~\cite{Jarrow2007uq} (the latter relying on the
very nice theory article~\cite{Sezer2007}). More generally, filtration
shrinkage appears naturally in models with restricted information, and
results such as those obtained in the present paper will be
instrumental for developing models of this type. This is discussed
further in Section~\ref{SApp}, which concludes the paper.

\subsection{Notation}
Let us now fix some notation that will be in force throughout the
paper. $(\Omega,\mathcal G,\mathbb G,P)$ is a filtered probability
space, where the filtration $\mathbb G=(\mathcal G_t)_{t\geq0}$ is the
right-continuous modification of a \emph{standard system}. That is,
$\mathcal G_t = \bigcap_{u>t} \mathcal G_u^o$, where each $\mathcal G_t^o$
is Standard Borel (see Parthasarathy~\cite{Parthasarathy1967},
Definition~V.2.2), and any decreasing sequence of atoms has a nonempty
intersection.\footnote{This means that if $(t_n)_{n\geq0}$ is a
nonnegative increasing sequence, $A_n \in\mathcal G^o_{t_n}$ is an
atom for each $n\geq1$, and $A_n \supset A_{n+1}$, then $\bigcap_n A_n
\neq\varnothing$.} We always assume that $\mathcal G=\mathcal G_\infty
=\bigvee_{t\geq0}\mathcal G_t$. A key example of a standard system is
the filtration generated by all right-continuous paths, allowed to
explode to an absorbing cemetery state in finite time, and with left
limits prior to the explosion time. This example is considered in
detail in \cite{Meyer1972}, and will re-appear in Section~\ref{SY} of
the present paper. Note that \emph{we do not augment $\mathbb G$ with
the $P$-nullsets}---but this does not cause any serious complications,
due to the following result which allows one to pass painlessly between
a filtration and its completion (see also \cite{CarrFisherRuf2012}
for a~discussion of this and related issues).

%
\begin{lemma} \label{Luh}
Let $R$ be a probability measure on $\mathcal G$, and denote by
$(\widebar{ \mathcal G}, \widebar{\mathbb G})$ the augmentation of
$(\mathcal G,\mathbb G)$ with respect to~$R$. Then:
\begin{longlist}[(ii)]
\item[(i)] Every $\widebar{\mathbb G}$~optional (predictable) process
is $R$-indistinguishable from a $\mathbb G$~optional (predictable) process.
\item[(ii)] Every right-continuous $(\mathbb G,R)$~martingale is a
$(\widebar{\mathbb G},R)$~martingale.
\end{longlist}
\end{lemma}

\begin{pf}
Part (i) is Lemma~7 in Appendix~1 of~\cite{DellacherieMeyer1982}.
Part~(ii) follows from Theorem~IV.3 in the same reference.
\end{pf}

Next, let $N$ be a local martingale on $(\Omega,\mathcal G,\mathbb G,
P)$ that is c\`adl\`ag, strictly positive and satisfies $N_0=1$,
$P$-a.s. Define stopping times
\[
\tau_n = n\wedge\inf \{ t\geq0\dvtx  N_t \ge n \}, \qquad
\tau = \lim_{n\to\infty} \tau_n.
\]
Since $N$ is a local martingale under $P$, and hence does not explode
in finite time, we have $P(\tau<\infty)=0$. However, there may be
$P$-nullsets on which $\tau$ is finite---in particular this is the case
when $N$ is a strict local martingale, as will become clear when we
discuss the F\"ollmer measure.

The reciprocal of $N$ will play a sufficiently important role that it
merits its own notation. We thus define a process $M$ by
%
%
\begin{equation}
\label{eqM} M_t = \frac{1}{N_t} {\mathbf1_{\{\tau> t\}}},
\end{equation}
whenever $N_t>0$, and $M_t=0$ otherwise ($N_t$ will never be zero under
any measure considered in the sequel).

Finally, note that $\mathcal G_{\tau-} = \bigvee_{n\geq1} \mathcal
G_{\tau_n}$, see, for instance, \cite{DellacherieMeyer1978}, Theorem~IV.56(d).

\section{The F\"ollmer measure} \label{SFM}

Following similar ideas as in~Delbaen and\break  Schachermayer \cite{DelbaenSchachermayer1995} and Pal and Protter~\cite
{PalProtter2010}, which originated with the paper by F\"ollmer~\cite
{Follmer1972} (who in turn was inspired by Doob~\cite{Doob1957}), we
can construct a new probability $Q_0$ on $\mathcal G_{\tau-}$ as
follows. For each $n\geq1$, the stopped process $N^{\tau
_n}=(N_{t\wedge
\tau_n})_{t\geq0}$ is a strictly positive uniformly integrable
martingale, so we may define a probability $Q_n\sim P$ on $\mathcal
G_{\tau_n}$ by $dQ_n = N_{\tau_n}\,dP$. The optional stopping theorem and
uniform integrability yield
\[
N_{\tau_n} = N_{\tau_n}^{\tau_{n+1}} = E^P \bigl[
N^{\tau
_{n+1}}_\infty \mid\mathcal G_{\tau_n} \bigr] =
E^P [ N_{\tau_{n+1}} \mid\mathcal G_{\tau_n} ].
\]
The measures $(Q_n)_{n\geq1}$ thus form a consistent family. Next, by
Remark~6.1 in the Appendix of~\cite{Follmer1972}, $(\mathcal G_{\tau
_n-})_{n\geq1}$ is a standard system, so Parthasarathy's extension
theorem (Theorem~V.4.2 in~\cite{Parthasarathy1967}) applies: there
exists a probability measure $Q_0$ on $\mathcal G_{\tau-}$ that
coincides with $Q_n$ on $\mathcal G_{\tau_n-}$, for each $n$.

\textit{From now on}, $Q_0$ \textit{will denote the measure on} $\mathcal G_{\tau-}$
\textit{obtained from} $P$ \textit{in this way}.

Here is the key point: $Q_0$ is only defined on $\mathcal G_{\tau-}$,
not on all of $\mathcal G$. There are typically many ways in which
$Q_0$ can be extended to a measure $Q$ on $\mathcal G$, and we will see
that the choice of extension is crucial in the context of filtration
shrinkage. In particular, the existence of an extension with certain
properties is intimately connected with the behavior of the optional
projection of $N$ (under $P$) onto smaller filtrations $\mathbb
F\subset\mathbb G$.

The following lemma shows that no matter which extension $Q$ one
chooses, $M$ defined in (\ref{eqM}) is always the density process
relative to~$P$. In particular it is a (true) $P$ martingale.

%
\begin{lemma} \label{Lfm1}
Suppose $Q$ is an extension of $Q_0$ to all of $\mathcal G$. Then, for
each $t\ge0$,
\[
M_t = \frac{dP}{dQ} \bigg|_{\mathcal G_t} \qquad Q\mbox{-a.s.}
\]
\end{lemma}

\begin{pf}
The argument is well-known. Fix $t\geq0$ and pick $A\in\mathcal G_t$.
Using that $M_t=0$ for $t\geq\tau$, monotone convergence and the fact
that $M_{t\wedge\tau_n}=\frac{dP}{dQ} |_{\mathcal G_{t\wedge
\tau
_n}}$ (which relies on the strict positivity of $N$), we obtain
\begin{eqnarray*}
E^Q [ M_t \mathbf1_A ] &=& E^Q [
M_t \mathbf1_{A\cap\{
\tau>t\}} ] = \lim_{n\to\infty}
E^Q [ M_t \mathbf1_{A\cap\{\tau_n>t\}} ]
\\
&=& \lim_{n\to\infty} E^Q [ M_{t\wedge\tau_n}
\mathbf1_{A\cap
\{\tau
_n>t\}} ] = \lim_{n\to\infty} P\bigl(A\cap\{
\tau_n>t\}\bigr)
\\
&=& P\bigl(A\cap\{\tau>t\}\bigr).
\end{eqnarray*}
Since $P(\tau>t)=1$, the right-hand side equals $P(A)$, as claimed.
\end{pf}

If $N$ is a strict local martingale under $P$, then $Q(\tau<\infty)>0$,
and vice versa. To see this, simply write
\[
Q(\tau>t) = E^Q[ M_t N_t ] =
E^P[ N_t ],
\]
which is strictly less than one for some $t>0$ if and only if $N$ is a
strict local martingale. Our focus will be on this case, and in
particular this means that $P$ and $Q$ cannot be equivalent. In fact,
they may even be singular, which is the case if $Q(\tau<\infty)=1$. On
the other hand, Lemma~\ref{Lfm1} guarantees that we always have \emph
{local absolute continuity}: for each $t$, $Q|_{\mathcal G_t }\ll
P|_{\mathcal G_t}$. ``Global'' absolute continuity, $Q\ll P$, holds
when $(M_t)_{t\geq0}$ is uniformly integrable under~$P$.

The following simple but useful result shows that although $N$ may
explode under $Q$, it does so continuously---it does not jump to infinity.

%
\begin{lemma} \label{LMtau}
On $\{\tau<\infty\}$, the equality $M_{\tau-}=0$ holds $Q_0$-a.s.
\end{lemma}

\begin{pf}
First, note that $\tau_n<\tau$, $Q_0$-a.s. Indeed, since $N^{\tau_n}$
is a martingale under $P$ and $\tau_n$ is bounded by construction,
\[
Q_0(\tau_n<\tau) = E^{Q_0}[M_{\tau_n}N_{\tau_n}]
= E^P[N_{\tau_n}] = 1.
\]
Now, on $\{N_{\tau-}<\infty$ and $\tau<\infty\}$ there exists a
(large) $n$ such that $\tau_n=\tau$. Hence
\[
Q_0 ( N_{\tau-}<\infty\mbox{ and }\tau< \infty ) \le
\sum_{n\ge1} Q_0 (\tau_n = \tau
) = 0.
\]
Therefore $Q_0( M_{\tau-}>0$ and $\tau< \infty)=0$,
as claimed.
\end{pf}

Let us mention that the construction of $P$ from $Q$ is
straightforward: assuming that $M$ is a $Q$~martingale, the measures
$P_n$ on $\mathcal G_n$ given by $dP_n = M_n\,dQ$ form a~consistent
family, extendable to a measure $P$ on $\mathcal G$ using
Parthasarathy's theorem. Local absolute continuity is immediate, and
``global'' absolute continuity holds when $M$ is uniformly integrable.
Note that $P$ only depends on the behavior of $Q$ on $\mathcal G_{\tau
-}$, since $P(\tau=\infty)=1$.

We finally comment on how the question of uniqueness has been treated
previously in the literature. In F\"ollmer's original paper~\cite
{Follmer1972}, a measure is constructed on the product space
$(0,\infty
]\times\Omega$, specifically on the predictable $\sigma$-field. This
measure assigns zero mass to the stochastic interval $(\tau,\infty]$,
which is key to obtaining uniqueness. On the other hand, neither~\cite
{DelbaenSchachermayer1995} nor~\cite{PalProtter2010} consider the
product space, but work directly on $\Omega$. However, $N$ is now taken
to be the coordinate process, with $+\infty$ as an absorbing state.
Hence there is ``no more randomness'' contained in the probability
space after $\tau$, which gives uniqueness of~$Q$. In the recent
paper~\cite{Kardarasetal2011}, Kardaras et al.~consider more general
probability spaces, and in particular discuss the question of
nonuniqueness. A construction of the F\"ollmer measure when the local
martingale $N$ may reach zero is discussed in~\cite{CarrFisherRuf2012}.

\section{Filtration shrinkage and a measure extension problem} \label{SMEP}

Consider now a filtration $\mathbb F=(\mathcal F_t)_{t\geq0}$ with
$\mathcal F_t\subset\mathcal G_t$, $t\geq0$, assumed to be the
right-continuous modification of a standard system. Again, completeness
is not assumed. The focus of this paper is on the object
\[
E^P [ N_t \mid\mathcal F_t ], \qquad t\ge0,
\]
interpreted as the optional projection of $N$ onto $\mathbb F$ (see below).

We suppose that $Q$ is an extension of $Q_0$ as discussed in
Section~\ref{SFM}. By Theorem~6 in Appendix~1 of~\cite
{DellacherieMeyer1982}, optional projections of $N$ and $M$ exist
under $P$~and~$Q$, respectively. When we write $E^P [ N_t \mid
\mathcal F_t  ]$ and $E^Q [ M_t \mid\mathcal F_t  ]$ we
always refer to these optional projections. Moreover, the projections
almost surely have c\`adl\`ag paths. This follows from the c\`adl\`ag
property of the optional projections onto the augmentation of $\mathbb
F$ (under $P$, resp., $Q$), together with Lemma~\ref{Luh} and the
uniqueness of the projection. A subtlety arises here: the optional
projection of $N$ under $P$ is unique up to a $P$-evanescent set.
However, this set need not be $Q$-evanescent. We will return to this
issue momentarily; see Remark~\ref{R1} below. First, however, we
introduce the following \emph{equivalent measure extension problem},
which turns out to be intimately related to properties of the optional
projections.

%
\begin{problem}[(Equivalent measure extension problem)]\label{MEP}
Given the probability $Q_0$ constructed in Section~\ref{SFM}, and the
subfiltration $\mathbb F\subset\mathbb G$, find a probability $Q$ on~$(\Omega,\mathcal G)$ such that:
\begin{longlist}[(ii)]
\item[(i)] $Q = Q_0$ on $\mathcal G_{\tau-}$;
\item[(ii)] The restrictions of $P$ and $Q$ to $\mathcal F_t$ are
equivalent for each $t\ge0$.
\end{longlist}
\end{problem}

%
\begin{remark}\label{R1}
The issue of $Q$-nonuniqueness of the optional projection of $N$ under
$P$ is resolved if $Q$ solves the equivalent measure extension problem.
Indeed, if $N'$ and $N''$ are two versions of $E^P[N_t\mid\mathcal F_t]$,
then for every $T\geq0$, $(N'_t)_{t\leq T}$ and $(N''_t)_{t\leq T}$
coincide on a set $A_T$ with $P(A_T)=1$. But $A_T\in\mathcal F_T$, so
$Q(A_T)=1$ as well. It follows that $N'=N''$~$Q$-a.s.
\end{remark}

%
\begin{remark}
If $N$ is a true martingale, then $Q_0(\tau=\infty)=1$, and the
equivalent measure extension problem has a trivial solution: take
$Q=Q_0$. Of course, for us the interesting case is when $N$ is a strict
local martingale.
\end{remark}

The following result clarifies the link between the equivalent measure
extension problem and filtration shrinkage.

%
\begin{lemma} \label{Lme1}
Fix $t\ge0$, and let $Q$ be any extension to $\mathcal G$ of $Q_0$.
Then the following are equivalent:
\begin{longlist}[(iii)]
\item[(i)] The restrictions of $P$ and $Q$ to $\mathcal F_t$ are equivalent.
\item[(ii)] $E^Q[M_t \mid\mathcal F_t] > 0$, $Q$-a.s.
\item[(iii)] $Q(\tau>t\mid\mathcal F_t) > 0$, $Q$-a.s.
\end{longlist}
If either of the above conditions holds, then
%
%
\begin{equation}
\label{eqme1} Q(\tau>t \mid\mathcal F_t) = E^Q
[M_t\mid\mathcal F_t ] E^P [N_t
\mid\mathcal F_t ], \qquad P\mbox{- and }Q\mbox{-a.s.}
\end{equation}
\end{lemma}

\begin{pf}
The equivalence of (i) and (ii) is immediate, since $E^Q[M_t\mid
\mathcal F_t]$ is the Radon--Nikodym density of $P|_{\mathcal F_t}$
with respect to $Q|_{\mathcal F_t}$. We now prove that (ii) and
(iii) are equivalent. To this end, let $A=\{E^Q[M_t\mid\mathcal
F_t]=0\} \in\mathcal F_t$. In the following, inclusions and equalities
are understood up to $Q$-nullsets. We have
\[
E^Q [ \mathbf1_AM_t ] = E^Q
\bigl[ \mathbf1_AE^Q[ M_t \mid \mathcal
F_t] \bigr] = 0,
\]
so $M_t=0$ on $A$. Hence $\tau\leq t$ on $A$, so
\[
E^Q \bigl[ \mathbf1_A Q(\tau> t \mid\mathcal
F_t) \bigr] = Q \bigl( A\cap \{\tau>t\} \bigr) = 0
\]
and we deduce that $Q(\tau> t \mid\mathcal F_t) = 0$ on $A$. The
reverse inclusion, $\{Q(\tau>t\mid\mathcal F_t) = 0\} \subset A$, is
proved similarly, and this gives (ii)${}\Longleftrightarrow{}$(iii). To\vspace*{1pt}
prove formula~(\ref{eqme1}), we use that $P(\tau>t)=1$, Bayes' rule
and the fact that $\frac{dP}{dQ} |_{\mathcal G_t}=M_t$ (Lemma~\ref{Lfm1}) to get
\begin{eqnarray*}
E^P [ N_t \mid\mathcal F_t ] &=& E^P \biggl[ \frac{1}{M_t} {\mathbf1_{\{\tau>t\}}} \biggm|
\mathcal F_t \biggr]
\\
&=& \frac{E^Q [ M_t (1/M_t) {\mathbf1_{\{\tau>t\}}} \mid
\mathcal F_t
] }{E^Q [ M_t \mid\mathcal F_t  ]} = \frac{Q(\tau>t \mid\mathcal F_t)}{E^Q [ M_t \mid\mathcal F_t
]}.
\end{eqnarray*}
This gives the desired conclusion.
\end{pf}

A solution $Q$ to the equivalent measure extension problem, when it
exists, leads to an interpretation of the finite variation part of the
$P$~optional projection onto $\mathbb F$ of the local martingale~$N$.
To see how, let us define
\[
Z_t = Q(\tau>t\mid\mathcal F_t). %
\]
This is an $(\mathbb F,Q)$~supermartingale, therefore it has a c\`adl\`
ag modification since $\mathbb F$ is right-continuous. We choose this
modification when defining $Z$. If in addition it is strictly positive,
it has a unique multiplicative Doob--Meyer decomposition
%
%
\begin{equation}
\label{eqZdec} Z_t = e^{-\Lambda_t}K_t,
\end{equation}
where $\Lambda$ is nondecreasing and predictable with $\Lambda_0=0$,
and $K$ is an $(\mathbb F,Q)$~local martingale with $K_0=1$, see
Theorem~II.8.21 in~\cite{Jacod2003fk}.

%
\begin{proposition} \label{PNdec}
Suppose $Q$ is a solution to the equivalent measure extension problem
(Problem~\ref{MEP}). Then $E^P[N_t\mid\mathcal F_t]$ is an $(\mathbb
F,P)$~supermartingale, with multiplicative decomposition
\[
E^P[N_t\mid\mathcal F_t] = e^{-\Lambda_t}
U_t,
\]
where $\Lambda$ is as in (\ref{eqZdec}) and $U$ is an $(\mathbb
F,P)$~local martingale. It is a true martingale provided $K$ in (\ref
{eqZdec}) is a true $(\mathbb F,Q)$ martingale.
\end{proposition}

\begin{pf}
If $Q$ solves the equivalent measure extension problem, Lemma~\ref{Lme1} implies that $Z$ is strictly positive, so that the
decomposition (\ref{eqZdec}) exists. It also implies that
\[
E^Q[M_t\mid\mathcal F_t] e^{\Lambda_t}
E^P[N_t\mid\mathcal F_t] = K_t
\]
is an $(\mathbb F,Q)$~local martingale. Since $E^Q[M_t\mid\mathcal
F_t]=\frac{dP}{dQ} |_{\mathcal F_t}$ it follows that $e^{\Lambda_t}
E^P[N_t\mid\mathcal F_t]$ is an $(\mathbb F,P)$~local martingale, and a
true martingale if $K$ is. Denoting this process by $U$ yields the
claimed decomposition.
\end{pf}

%
\begin{remark}
The fact that $E^P[N_t\mid\mathcal F_t]$ is an $(\mathbb
F,P)$~supermartingale also follows from Theorem~2.3 in~\cite
{FollmerProtter2011}. Moreover, it is of Class~(DL) whenever $U$ is
a~martingale, and by Proposition~\ref{PNdec} this holds if $K$ is a
martingale. A simple sufficient condition for this is that $\Lambda$
does not increase too rapidly, in the sense that $E^Q[e^{\Lambda
_t}]<\infty$ for each $t\ge0$. Indeed, in this case $E^Q[\sup_{s\le t}
K_s]<\infty$ since $Z\le1$, implying the martingale property.
\end{remark}

The following corollary is simple but nonetheless informative, since it
shows that the equivalent measure extension problem certainly does not
always have a~solution.

%
\begin{corollary} \label{Cns}
Suppose $N$ is a strict $(\mathbb G,P)$~local martingale. If
$E^P[N_t\mid\mathcal F_t]$ is again an $(\mathbb F,P)$~local
martingale, then the equivalent measure extension problem has no solution.
\end{corollary}

\begin{pf}
Suppose a solution exists. Then, since $E^P[N_t\mid\mathcal F_t]$ is a
local martingale, the process $\Lambda$ in Proposition~\ref{PNdec} is
identically zero, so that $K$ is bounded and hence a true
martingale.\vadjust{\goodbreak}
Therefore $E^P[N_t\mid\mathcal F_t]=U_t$ is a true martingale by
Proposition~\ref{PNdec}. It follows that $E^P[N_t]=E^P[E^P[N_t\mid
\mathcal F_t]]=1$ for all $t\geq0$, contradicting that $N$ is a strict
local martingale.
\end{pf}

We can now establish our first main result. It shows that the finite
variation part $\Lambda$ appearing when $N$ is projected onto the
smaller filtration can be interpreted as the \emph{predictable
compensator of} $\tau$, viewed in the appropriate filtration. The key
step is an application of the Jeulin--Yor theorem from the theory of
filtration expansions.

%
\begin{theorem} \label{Tcomp}
Let $\mathbb F^{\tau}$ be the progressive expansion of $\mathbb F$ with
$\tau$, that is, the smallest filtration that contains $\mathbb F$,
satisfies the usual hypotheses (with respect to~$Q$) and makes $\tau$ a
stopping time. If $Q$ solves the equivalent measure extension problem, then:
\begin{longlist}[(ii)]
\item[(i)] the process
\[
{\mathbf1_{\{\tau\leq t\}}} - \Lambda_{t\wedge\tau} %
\]
is an $(\mathbb F^{\tau},Q)$ uniformly integrable martingale, where
$\Lambda$ is as in (\ref{eqZdec});
\item[(ii)] $\tau$ is not $\mathbb F^{\tau}$-predictable, provided
$Q(\tau<\infty)>0$.
\end{longlist}
\end{theorem}

\begin{pf}
The proof uses stochastic integration, which can be developed without
assuming the usual hypotheses; see, for instance, Chapter~I.4 in~\cite
{Jacod2003fk}. Alternatively, one may apply Lemma~\ref{Luh} to first
pass to the $Q$-completion $\widebar{\mathbb F}$ of $\mathbb F$
without losing the semimartingale property of any of the processes
involved, carry out the computations there and then go back to $\mathbb
F$ at the cost of changing things on a \mbox{$Q$-}nullset.

The integration by parts formula yields
\[
Z_t = 1 + \int_0^t
e^{-\Lambda_{s-}}\,dK_s + \bigl[e^{-\Lambda},K\bigr]_t
- \int_0^t e^{-\Lambda_{s-}} K_{s-}\,d
\Lambda_s. %
\]
By Yoeurp's lemma (\cite{DellacherieMeyer1982}, Theorem~VII.36),
$[e^{-\Lambda},K]$ is a local martingale, so we have the additive
Doob--Meyer decomposition $Z_t = \mu_t - a_t$, where
\[
\mu_t = 1+\int_0^t
e^{-\Lambda_{s-}} \,dK_s + \bigl[e^{-\Lambda},K\bigr]_t
\quad\mbox{and}\quad a_t = \int_0^t
Z_{s-} \,d\Lambda_s. %
\]
The Jeulin--Yor theorem (see Theorem~1.1 in~\cite{GuoZeng2008}, or
the original paper~\cite{JeulinYor1978}), which is applicable in view
of Lemma~\ref{Luh}, shows that the process
\[
{\mathbf1_{\{\tau\leq t\}}} - \int_0^{t\wedge\tau}
\frac{1}{Z_{s-}}\,da_s %
\]
is an $(\mathbb F^{\tau},Q)$ martingale, and indeed uniformly
integrable since it is the martingale part of the Doob--Meyer
decomposition of the Class~(D) submartingale ${\mathbf1_{\{\tau\le
\cdot\}}}$.
Substituting for $da_s$ yields~(i).

To prove (ii), assume for contradiction that there is a strictly
increasing sequence of $\mathbb F^{\tau}$~stopping times $\rho_n$ such
that $\lim_n \rho_n = \tau$. By the lemma on page 370 in~\cite
{Protter2005}, there are $\mathbb F$~stopping times $\sigma_n$ such
that $\sigma_n\wedge\tau=\rho_n\wedge\tau$. But since $\rho
_n<\tau$,
this yields $\sigma_n=\rho_n$. It follows that $\tau$ is $Q$-a.s.~equal
to an $\mathbb F$~stopping time, implying that
\[
Q(\tau>t\mid\mathcal F_t) = {\mathbf1_{\{\tau>t\}}} \qquad Q
\mbox{-a.s.} %
\]
By Lemma~\ref{Lme1} this contradicts the assumption that $Q$ solves
the equivalent measure extension problem, since by hypothesis $Q(\tau
<\infty)>0$.
\end{pf}

The significance of Theorem~\ref{Tcomp} is that it shows \emph{when}
the $(\mathbb F,P)$~supermartingale $E^P[N_t\mid\mathcal F_t]$ loses
mass: it happens exactly when the compensator of $\tau$ increases, that
is, when there is an increased probability, conditionally on $\mathbb
F$, that $\tau$ has already happened. This corresponds to a kind of
smoothing over time of the sets $\{\tau\leq t\}$ when we pass to the
smaller filtration~$\mathbb F$. This smoothing is necessary to make the
restrictions of $P$ and $Q$ equivalent, since $\{\tau\leq t\}$ is
$P$-null but not necessarily $Q$-null.

\section{The finite variation term in a topological setting} \label{SY}

In this section we specialize the previous setup as follows. Let $E$ be
a locally compact topological space with a countable base, and define
$E_\Delta= E\cup\{\Delta\}$, where $\Delta\notin E$ is an isolated
point. We take $\Omega$ to be all right-continuous paths $\omega\dvtx \mathbb R_+ \to E_\Delta$ that are absorbed at $\Delta$
[i.e., if $\omega(s)=\Delta$ then $\omega(t)=\Delta$ for all $t\ge s$] and have
left limits on $(0,\zeta(\omega))$, where the \emph{absorption time}
$\zeta$ is defined by
\[
\zeta(\omega):=\inf\bigl\{t\ge0\dvtx  \omega(t)=\Delta\bigr\}.
\]
Let $Y_t(\omega)=\omega(t)$ be the coordinate process, and define
$\mathcal G_t^o=\sigma(Y_s\dvtx s\le t)$. Then $\mathbb G^o=(\mathcal
G_t^o)_{t\ge0}$ is a standard system; see the Appendix in \cite
{Follmer1972}. We let $\mathbb G$ be the right-continuous modification
of $\mathbb G^o$, and $\mathcal G=\bigvee_{t\ge0}\mathcal G_t$.

Next, consider a function $h\dvtx E_\Delta\to[0,\infty)$ that is continuous
on $E$ and satisfies $h(Y_0)=1$ $P$-a.s. (In particular, the measure
$P$ is such that, almost surely, $Y$ starts at a point where $h$ equals
one.) Define stopping times $\tau_n=n\wedge\inf\{t\ge0\dvtx h(Y_t)\le
1/n\}$
and $\tau=\lim_{n\to\infty}\tau_n$. We assume that the $P$~local
martingale $N$ is given by
\[
N_t = \frac{1}{h(Y_t)}{\mathbf1_{\{\tau>t\}}}.
\]
Note how this imposes restrictions on the interplay between~$P$
and~$h$: they have to be such that $N$ is indeed a local martingale.
Note also that given this setup, the definitions of $\tau_n$ and $\tau$
are consistent with those given in Section~\ref{SB}. Furthermore, we
let $M$ be given by~(\ref{eqM}), and $Q_0$ as in Section~\ref{SFM}.

To describe the smaller filtration $\mathbb F$, let $D$ be a metrizable
topological space, and let
\[
\pi\dvtx  E \to D
\]
be a continuous map. We define $D_\Delta= D\cup\{\Delta\}$ (assuming
without loss of generality that $\Delta\notin D$), and set $\pi
(\Delta
)=\Delta$. If $d(\cdot,\cdot)$ is a metric on $D$, we extend it
$D_\Delta$ by setting $d(x,\Delta)=d(\Delta,x)=\infty$ for $x\in D$,
and $d(\Delta,\Delta)=0$. Next, define a $D_\Delta$-valued process
$X$ by
\[
X_t = \pi(Y_t), \qquad t\ge0.
\]
It is clear that $X$ is $\mathbb G$-adapted. The filtration $\mathbb
F=(\mathcal F_t)_{t\ge0}$, given by
\[
\mathcal F_t = \bigcap_{u>t}
\sigma(X_s\dvtx  s\le u),
\]
is therefore a subfiltration of $\mathbb G$, right-continuous, but not
augmented. The structure imposed by the above conditions (and the
flavor of the main theorem below) is primarily of a topological nature,
which motivates the title of this section.

Recall the multiplicative decomposition $E^P[N_t\mid\mathcal F_t] =
e^{-\Lambda_t} U_t$ of the positive $(\mathbb F,P)$ supermartingale
$E^P[N_t\mid\mathcal F_t]$. The finite variation part $\Lambda$ is
related to $\tau$ by Proposition~\ref{PNdec}, provided the equivalent
measure extension problem has a solution. In the particular setting of
the present section, we can say the following about the points of
increase of~$\Lambda$:

%
\begin{theorem} \label{TdL}
Assume that $Q$ is a solution to the equivalent measure extension
problem, and let $\Lambda$ be as in~(\ref{eqZdec}). Then the random
measure $d\Lambda_t$ is supported on the set $\{t\dvtx  X_{t-} \in
\overline{D_0}\}$, where $\overline{D_0}$ is the closure in $D$ of
\[
D_0 = \pi\circ h^{-1}\bigl(\{0\}\bigr) = \bigl\{x\in D\dvtx  x=
\pi(y)\mbox{ for some } y\in E \mbox{ with } h(y)=0 \bigr\}.
\]
\end{theorem}

The proof requires two lemmas.

%
\begin{lemma} \label{LdL1}
We have $\pi(Y_{\tau-}) \in D_0$ on $\{\tau<\infty\}$, $Q_0$-a.s.
\end{lemma}

\begin{pf}
To show that $\pi(Y_{\tau-}) \in D_0$, one must find $y\in E$ with
$h(y)=0$ such that $\pi(y)=\pi(Y_{\tau-})$. But $h(Y_{\tau
-})=M_{\tau
-}=0$ on $\{\tau<\infty\}$ by Lemma~\ref{LMtau}, so we may take
$y=Y_{\tau-}$.
\end{pf}

%
\begin{lemma} \label{LdL2}
For any $\mathbb F$~stopping time $\rho$, the equality $Z_\rho=Q(\tau
>\rho\mid\mathcal F_\rho)$ holds on $\{\rho<\infty\}$, $Q$-a.s.
\end{lemma}

\begin{pf}
We need to show that $E^Q[Z_\rho\mathbf1_{A\cap\{\rho<\infty\}
}]\vspace*{2pt}=Q(A\cap\{
\tau>\rho\})$ for every\vadjust{\goodbreak} $\mathbb F$-stopping time $\rho$ and every
$A\in
\mathcal F_\rho$. This clearly holds when $\rho$ is constant. Suppose
now that $\rho$ is of the form
%
%
\begin{equation}
\label{eqL2dL} \rho= \sum_{i=1}^n
t_i \mathbf1_{A_i},
\end{equation}
where $t_i\in[0,\infty]$, $A_i\in\mathcal F_{t_i}$, and the $A_i$
constitute a partition of $\Omega$. Then
\begin{eqnarray*}
E^Q[Z_\rho\mathbf1_A] &=& \sum
_{i=1}^n E^Q [Z_{t_i}\mathbf
1_{A_i\cap
A} ]
\\
&=& \sum_{i=1}^n Q \bigl(A\cap
A_i\cap\{\tau>t_i\} \bigr)
\\
&=& Q \bigl(A\cap\{\tau>\rho\} \bigr),
\end{eqnarray*}
where the second equality used that $A_i\cap A\in\mathcal F_{t_i}$ and
that the result holds for constant times. Finally, let $\rho_n$ be a
decreasing sequence of stopping times of the form (\ref{eqL2dL}) with
$\lim_n\rho_n=\rho$. Right-continuity together with bounded convergence
and the result applied to $\rho_n$ now yields the statement of the lemma.
\end{pf}

\begin{pf*}{Proof of Theorem~\ref{TdL}}
Let $0<\rho\le\sigma$ be bounded $\mathbb F$~stopping times such that
$X_-\notin D_0$ on $[\rho, \sigma)$. We claim that $\Lambda_\sigma
-\Lambda_\rho=0$, $Q$-a.s. To prove this, first write
%
%
\begin{equation}
\label{eqdL0} \Lambda_\sigma- \Lambda_\rho= (
\Lambda_\sigma- \Lambda_\rho ){\mathbf1_{\{\tau \le\sigma\}}} + (
\Lambda_{\sigma\wedge\tau} - \Lambda_{\rho\wedge\tau}){\mathbf1_{\{\sigma<\tau\}}}.
\end{equation}
By continuity of $\pi$ and the choice of $\rho$ and $\sigma$, we have
\[
\pi(Y_{\tau-}) = X_{\tau-} \notin D_0\qquad\mbox{on }
\{\rho<\tau \le\sigma\}.
\]
But according to Lemma~\ref{LdL1}, $\pi(Y_{\tau-})\in D_0$ on $\{
\tau
<\infty\}$, $Q$-a.s., so we get
%
%
\begin{equation}
\label{eqdL1} Q (\rho<\tau\le\sigma ) \le Q \bigl(\pi(Y_{\tau-}) \notin
D_0, \tau<\infty \bigr) = 0.
\end{equation}
Next, consider the filtration $\mathbb F^{\tau}$ described in
Theorem~\ref{Tcomp}. By that theorem,
\[
{\mathbf1_{\{\tau\leq t\}}} - \Lambda_{t\wedge\tau} %
\]
is an $(\mathbb F^{\tau},Q)$~martingale. Since $\rho$ and $\sigma$ are
also $\mathbb F^{\tau}$~stopping times, the martingale property and the
optional sampling theorem, together with~(\ref{eqdL1}), yield
\[
E^Q [ \Lambda_{\sigma\wedge\tau} - \Lambda_{\rho\wedge\tau
} ] =
E^Q [ {\mathbf1_{\{\rho<\tau\leq\sigma\}}} ] = 0.
\]
Since $\Lambda$ is nondecreasing, we deduce that $\Lambda_{\sigma
\wedge
\tau} - \Lambda_{\rho\wedge\tau} = 0$, $Q$-a.s. Using this in the
decomposition (\ref{eqdL0}), we obtain
\[
\Lambda_\sigma- \Lambda_\rho= (\Lambda_\sigma-
\Lambda_\rho ){\mathbf1_{\{\tau \le\rho\}}}.
\]
This implies that $\tau\le\rho$ on the $\mathcal F_\sigma$-measurable
set $\{\Lambda_\sigma- \Lambda_\rho> 0 \}$. In conjunction with
Lemma~\ref{LdL2}, this gives the equalities
\[
Z_\sigma{\mathbf1_{\{\Lambda_\sigma- \Lambda_\rho> 0\}}} = Q \bigl( \{\tau> \sigma\}\cap\{
\Lambda_\sigma- \Lambda_\rho> 0 \} \mid\mathcal
F_\sigma \bigr) = 0, \qquad Q\mbox{-a.s.}
\]
But $Q$ solves the equivalent measure extension problem, so $Z$ is
strictly positive, $Q$-a.s. Therefore $Q(\Lambda_\sigma- \Lambda
_\rho
> 0)=0$, and we have finally proved our claim that $\Lambda_\sigma
-\Lambda_\rho=0$, $Q$-a.s.

Now, choose a metric $d(\cdot,\cdot)$ on $D$ compatible with its
topology. For any subset $A\subset D$ and any $x\in D$, define the
distance from $x$ to $A$ by
\[
\operatorname{dist}(x,A) = \inf\bigl\{ d\bigl(x,x'\bigr)\dvtx
x'\in A\bigr\}.
\]
It is easy to check that $\operatorname{dist}(\cdot,A)$ is continuous (even
Lipschitz), and in particular measurable. For each rational number
$r>0$ and natural number $n>r$, define stopping times
\begin{eqnarray*}
\rho_r &=& \cases{r, &\quad if $\operatorname{dist}(X_{r-},D_0)
> 0$,
\cr
\infty, &\quad otherwise,}
\\
\rho_{r,n} &=& n\wedge\rho_r,
\\
\sigma_r &=& n\wedge\inf\bigl\{ t> \rho_{r,n}\dvtx
\operatorname{dist}(X_{t-}, D_0) = 0\bigr\}.
\end{eqnarray*}
Then the stopping times $\rho_{r,n}$ and $\sigma_{r,n}$ are all
bounded, and it is a simple matter to check the inclusion
\[
[\rho_{r,n},\sigma_{r,n})\subset\bigl\{\operatorname{dist}(X_-,D_0)>0
\bigr\}.
\]
Moreover, if for some $(t,\omega)$ with $t>0$ it holds that $\operatorname{dist}(X_{t-}(\omega), D_0)>0$, then by left-continuity of $X_-(\omega
)$ and continuity of the distance function, there is a rational $r>0$
such that $r\le t$ and for all $s\in[r,t]$ we have $\operatorname{dist}(X_{s-},D_0)>0$. Thus for any $n>t$, we have $(t,\omega)\in
[\rho
_{r,n},\sigma_{r,n})$, and we deduce
\[
\mathop{\bigcup_{r\in\mathbb Q, r>0}}_{n\in\mathbb N, n>r} [\rho
_{r,n}, \sigma_{r,n}) = \bigl\{ \operatorname{dist}(X_-,D_0)>0
\bigr\}.
\]
By the first part of the proof, $d\Lambda_t$ does not charge any of the
countably many intervals in the union on the left-hand side. It follows
that $d\Lambda_t$ is supported on the set $\{\operatorname{dist}(X_-, D_0) =
0\}$, which coincides with $\{X_-\in\overline{D_0}\}$.
\end{pf*}

%
\begin{remark}
Since $h$ is continuous and $D_0=\pi\circ h^{-1}(\{0\})$, $D_0$ is a
closed set in $D$ if $\pi$ is a closed map. An example is when $\pi$ is
a linear map on $\mathbb R^q$, and $D=\pi(\mathbb R^q)$. This case is
discussed in Section~\ref{SMEPsol}.
\end{remark}

We can now give a simple sufficient condition for $\Lambda$ to have
singular paths, as in the example studied by F\"ollmer and
Protter~\cite{FollmerProtter2011} that was mentioned in the \hyperref[SB]{Introduction}.

%
\begin{corollary} \label{CJY}
Assume $D$ is a subset of~$\mathbb R^k$ for some $k$, and that the law
of~$X_t$ under $Q$ admits a density for almost every $t> 0$. Then, if
$\overline{D_0}$ is a nullset in $\mathbb R^k$, the paths of $\Lambda$
are singular.
\end{corollary}

\begin{pf}
Since $X_s$ has a density for almost every $s$ and $\overline{D_0}$ is
a nullset, Fubini's theorem yields $E^Q[ \int_0^t {\mathbf1_{\{X_s\in
\overline{D_0}\}}} \,ds ] = \int_0^t Q (X_s\in\overline{D_0} ) \,ds = 0$.
Hence $\int_0^t {\mathbf1_{\{X_s\in\overline{D_0}\}}}\,ds=0$, $Q$-a.s.
Thus $\{t\dvtx
X_t\in\overline{D_0}\}$ is a nullset $Q$-a.s., and it contains the
support of $d\Lambda_t$ by Theorem~\ref{TdL}. This proves the claim.
\end{pf}

We finish this section with a result intended to emphasize the
distinction between $\zeta$, the absorption time of the coordinate
process $Y$, and the explosion time $\tau$ of the process~$N$.

%
\begin{proposition} \label{Pncmeps}
The following statements hold:
\begin{longlist}[(ii)]
\item[(i)] Let $Q$ be any extension of $Q_0$ to all of $\mathcal G$.
Then $\tau\le\zeta$ on $\{\tau<\infty\}$, $Q$-a.s.
\item[(ii)] If $Q$ is a solution to the equivalent measure extension
problem and $\tau<\infty$ on $\{\zeta<\infty\}$, $Q$-a.s., then
$Q(\zeta
=\infty)=1$.
\end{longlist}
\end{proposition}

\begin{pf}
Since the coordinate process stops at $\zeta$, it is clear that
$\mathcal G_\infty= \mathcal G_\zeta$. Hence for any stopping time
$\sigma$, $\mathcal G_\sigma= \mathcal G_\sigma\cap\mathcal G_\zeta
\subset\mathcal G_{\sigma\wedge\zeta}\subset\mathcal G_\sigma$, and
thus $\mathcal G_{\sigma\wedge\zeta}=\mathcal G_\sigma$. Applying this
with $\sigma=t$, for any $t\ge0$, we get
\[
M_{t\wedge\zeta} = E^Q [ M_t \mid\mathcal
G_{t\wedge\zeta
} ] = E^Q [ M_t \mid\mathcal
G_t ] = M_t,
\]
showing that $M$ is $Q$-a.s. constant after $\zeta$ (note that this
holds for any martingale). Now, on $\{\zeta<\tau\}$ we have $\inf_{0\le
t\le\zeta}M_t > 0$, and since $M$ is constant after $\zeta$ we have
$\inf_{t\ge0}M_t>0$. Hence $\tau=\infty$, and we deduce~(i).

To prove (ii), first note that $X_\zeta=\pi(Y_\zeta)=\pi(\Delta
)=\Delta$, and that for $t<\zeta$, $X_t\in D$ so that $X_t\ne\Delta$.
The absorption time can therefore alternatively be written
\[
\zeta= \inf\{t\ge0\dvtx  X_t = \Delta\},
\]
showing that $\zeta$ is in fact an $\mathbb F$~stopping time. Our
hypothesis says that $\tau<\infty$ on $\{\zeta<\infty\}$. Hence, by
part~(i) above, $\tau\le\zeta$ on $\{\zeta<\infty\}$. But since
Lemma~\ref{LdL2} implies that $Z_\zeta=Q(\tau>\zeta\mid\mathcal
F_\zeta
)$ on this set, we deduce that $Z_\zeta=0$ on $\{\zeta<\infty\}$. Now,
$Q$ solves the equivalent measure extension problem, so in order to
avoid a contradiction we must have $Q(\zeta=\infty)=1$.
\end{pf}

%
\begin{remark}
If $N$ itself is the coordinate process, then $\tau$ and $\zeta$
coincide, as is the case, for example, in \cite{DelbaenSchachermayer1995}. In this case part (ii) of the above
proposition implies that the equivalent measure extension problem lacks
a solution \emph{for any subfiltration}~$\mathbb F$ of the type
discussed in this section. At first glance, this seems to imply that
the proposition is incorrect: let, for instance, $\mathbb F$ be the
trivial filtration---then $P$ itself is a solution to the equivalent
measure extension problem. The issue here is that the trivial
filtration is not of the type introduced above, since we assumed that
$\pi(\Delta)=\Delta\ne\pi(y)$ for $y\in E$. In particular, $\zeta
$ is
not a stopping time for the trivial filtration, and this breaks the
proof of part (ii). On the other hand, part~(i) remains correct
even if we allow $\pi(\Delta)$ to lie in $D$, and also part (ii)
remains correct as long as we additionally assume that $\zeta$ is an
$\mathbb F$~stopping time.
\end{remark}

\section{Solving the equivalent measure extension problem} \label{SMEPsol}

So far we have \emph{assumed} that the equivalent measure extension
problem has a solution. In this section we specialize the setup from
Section~\ref{SY}, imposing further assumption that enable us to prove
the existence of a particular solution, and to describe this solution
explicitly. This is done in Section~\ref{SSMEP1}. Some examples where
the main result (Theorem~\ref{Tmain11} below) applies are then
discussed in Section~\ref{Sex}. The symbol $|\cdot|$ denotes the
usual Euclidean norm, and $\nabla$ is the gradient.

\subsection{Linear shrinkage in a Brownian setting} \label{SSMEP1}

We make the following assumptions, within the framework described in
Section~\ref{SY}:
\begin{itemize}
\item$E=\mathbb R^q$, some $q\in\mathbb N$.
\item$P$ is Wiener measure, turning the coordinate process $Y$ into
$q$-dimensional Brownian motion (possibly starting from $Y_0\ne0$).
\item$h$ is such that $\frac{1}{h}$ is harmonic on $\mathbb R^q
\setminus E_0$, where we define
\[
E_0=h^{-1}\bigl(\{0\}\bigr).
\]
\item$\pi\dvtx E\to E$ is linear, and we set $D = \pi(\mathbb R^q)$ and
$p=\dim D=\operatorname{rank}\pi$. We assume $p<q$, since otherwise
we have $\mathbb
F=\mathbb G$, in which case the equivalent measure extension problem
has a solution precisely when $N$ is already a martingale under~$P$.
\end{itemize}

The main result is the following.

%
\begin{theorem} \label{Tmain11}
Consider the setup just described, and assume furthermore that $h$
satisfies the following conditions:
%
%
\begin{eqnarray}
\label{eqlbm1} t &\mapsto& E^P \biggl[ \frac{ |\nabla\ln h(Y_t)| }{h(Y_t)} \biggr]
\qquad\mbox{is locally bounded on } [0,\infty),
\\
(t,x) &\mapsto& E^P \biggl[ \frac{|\pi(\nabla\ln h(Y_t))|
}{h(Y_t)}\biggm|
\pi(Y_t)=x \biggr]
\nonumber\\[-10pt] \label{eqlbm2}  \\[-10pt]
\eqntext{\mbox{is locally bounded on } (0,\infty )\times D,}
\end{eqnarray}
where the right-hand side of~(\ref{eqlbm2}) should be understood in
the sense of regular conditional probabilities. Then the equivalent
measure extension problem has a solution~$Q$ with the property that
%
%
\begin{equation}
\label{eqYQ} \qquad W = \biggl( Y_t - Y_0 + \int
_0^{t\wedge\tau} \nabla\ln h(Y_s)\,ds
\biggr)_{t\ge0}\qquad\mbox{is $Q$-Brownian motion},
\end{equation}
where the integral is well-defined and finite for each $t\ge0$,~$Q$-a.s.
\end{theorem}
%

%
\begin{remark}
The role of condition~(\ref{eqlbm1}) is primarily to ensure that the
optional projection of $Y$ under $Q$ can be computed in a reasonable
way. Moreover, since trivially $\pi$ is a bounded operator, (\ref
{eqlbm1}) also implies that the conditional expectation in~(\ref
{eqlbm2}) is finite for each $(t,x)\in(0,\infty)\times D$. The role of
condition~(\ref{eqlbm2}) is to ensure that $\mathbb F$ is small enough
for the projection operation to induce sufficient smoothing. In
particular, if $D$ is zero-dimensional, so that $\mathbb F$ is the
trivial filtration, then~(\ref{eqlbm2}) automatically holds.
\end{remark}

%
\begin{remark}
Unfortunately the assumptions of Theorem~\ref{Tmain11} are quite
restrictive. While they do allow us to treat the example by F\"ollmer
and Protter mentioned in the \hyperref[SB]{Introduction}, a major open problem for
future research is to find more general conditions under which the
equivalent measure extension problem can be solved.
\end{remark}

%
\begin{remark}
Theorem~\ref{Tmain11} is a closely related to Doob's $h$-transform.
Indeed, one can view $P$ as being obtained from $Q$ by conditioning $Y$
never to hit the zero set of $h$. Note, however, that $Y$ is not
Markovian under $Q$ due to the presence of $\tau$.
\end{remark}

The rest of this section is devoted to the proof of Theorem~\ref{Tmain11}. The strategy can be summarized as follows: we first exhibit
an extension $Q$ of $Q_0$ for which (\ref{eqYQ}) holds. Then we
describe the law of $X=\pi(Y)$ under $P$ and under $Q$. Finally, this
description is used to show that the laws are locally equivalent. Since
$X$ generates $\mathbb F$ this yields the result. We now turn to the
details, which are carried out through a~sequence of lemmas.

%
\begin{lemma} \label{LQ0fin}
Assume that (\ref{eqlbm1}) is satisfied. Then the inequality
%
%
\begin{equation}
\label{eqQ0fin} \int_0^t E^{Q_0}
\bigl[ \bigl|\nabla\ln h(Y_s)\bigr| {\mathbf1_{\{s<\tau\}}} \bigr] \,ds <
\infty
\end{equation}
holds for every $t\ge0$. Consequently, there is an extension $Q$ of
$Q_0$ for which (\ref{eqYQ}) holds.
\end{lemma}
%

\begin{pf}
We have
\[
E^{Q_0} \bigl[ \bigl|\nabla\ln h(Y_t)\bigr| {\mathbf1_{\{t<\tau\}}}
\bigr] = E^P \biggl[\frac{1}{h(Y_t)} \bigl|\nabla\ln h(Y_t)\bigr|
\biggr].\vadjust{\goodbreak}
\]
By (\ref{eqlbm1}), the right-hand side is locally integrable in $t$ on
$[0,\infty)$, which implies (\ref{eqQ0fin}). We may therefore define
an $E_\Delta$-valued process $W$ by
\[
W_t = Y_t - Y_0 + \int_0^{t\wedge\tau}
\nabla\ln h(Y_s) \,ds, \qquad t\ge0,
\]
using (\ref{eqQ0fin}) to see that the integral on the right-hand side
is well defined and finite. Now, for each~$n$, $N^{\tau_n}$ is the
density process of the restriction of $Q_0$ to $\mathcal G_{\tau_n}$
with respect to $P$. (Recall that $\tau_n$ is the minimum of $n$ and
the first time $N_t$ hits level~$n$.) We observe that, by It\^o's formula,
\[
N_t = \frac{1}{h(Y_t)} = 1 - \int_0^t
N_s \nabla\ln h(Y_s) \,dY_s, \qquad t<\tau,
\]
so that an application of Girsanov's theorem yields that $(W_{t\wedge
\tau_n}\dvtx t\ge0)$ is a local martingale for each $n$. Since $\langle
W^i,W^j\rangle_{t\wedge\tau_n} = (t\wedge\tau_n)\delta_{ij}$, it
is in
fact a martingale behaving like stopped Brownian motion. A standard
argument based on Doob's up- and downcrossing inequalities then shows
that the limit $\lim_{t\uparrow\tau}W_t$ exists in $\mathbb R^q$ on
$\{
\tau<\infty\}$, $Q_0$-a.s. As a consequence, $Y_{\tau-}$ also exists on
$\{\tau<\infty\}$, and is different from $\Delta$. We now simply choose
the law $Q$ so that $Y_{\tau}=Y_{\tau-}$ and $(Y_{\tau+t}-Y_{\tau
}\dvtx t\ge
0)$ is Brownian motion.
\end{pf}

Since $Y-Y_0$ is Brownian motion under $P$, it is clear that the same
holds for $X-X_0=\pi(Y-Y_0)$, but with a possibly different quadratic
covariation depending on~$\pi$. The following lemma describes what
happens under~$Q$.

%

%
\begin{lemma} \label{LXdec}
Assume that (\ref{eqlbm1}) is satisfied, and let $Q$ be an extension
of~$Q_0$ for which (\ref{eqYQ}) holds (it exists by Lemma~\ref{LQ0fin}).
The process $X$ can then be decomposed as
\[
X_t = X_0 + B_t + \int_0^t
\theta_s \,ds\qquad\mbox{for all } t\ge0, Q\mbox{-a.s.},
\]
where $B$ is $(\mathbb F,Q)$~Brownian motion (with the same quadratic
covariation as~$X$), and $\theta_t$ satisfies, for every $t\ge0$,
\[
\theta_t = E^Q \bigl[ \pi\bigl( \nabla\ln
h(Y_t) \bigr) {\mathbf1_{\{\tau>t\}
}} \mid\mathcal F_t
\bigr]\qquad Q\mbox{-a.s.}\quad\mbox{and}\quad \int_0^tE^Q
\bigl[ |\theta_s| \bigr] \,ds < \infty.
\]
\end{lemma}

\begin{pf}
Due to Lemma~\ref{LQ0fin}, the optional projection of $\pi(\nabla
\ln
h(Y_t)){\mathbf1_{\{\tau>t\}}}$ onto $\mathbb F$ is well defined under
$Q$. Denoting
this optional projection by $\theta$ it is clear that the given
expression for $\theta$ and the integrability statement are correct.
From~(\ref{eqYQ}), the definition of $X_t$ and the linearity of $\pi$
we obtain
\begin{eqnarray*}
X_t &=& E^Q \bigl[ \pi(Y_t) \mid\mathcal
F_t \bigr]
\\
&=& \pi(Y_0) + E^Q \bigl[ \pi(W_t) \mid
\mathcal F_t \bigr] - E^Q \biggl[ \int
_0^t \pi\bigl(\nabla\ln h(Y_s)
\bigr) {\mathbf1_{\{s<\tau\}}} \,ds \biggm| \mathcal F_t \biggr]
\\
&=& X_0 + B_t - \int_0^t
\theta_s \,ds,
\end{eqnarray*}
where we define $B_t = E^Q [ \pi(W_t) \mid\mathcal F_t ] +
L_t$ with
\begin{eqnarray*}
L_t &=& E^Q \biggl[ \int_0^t
\pi\bigl(\nabla\ln h(Y_s)\bigr) {\mathbf1_{\{s<\tau
\}}} \,ds \biggm|\mathcal F_t \biggr]
\\
&&{}- \int_0^t E^Q \bigl[ \pi
\bigl(\nabla\ln h(Y_s)\bigr) {\mathbf1_{\{
s<\tau\}}} \mid \mathcal
F_s \bigr] \,ds.
\end{eqnarray*}
Suppose we know $B$ is a (local) martingale. Since its quadratic
covariation coincides with that of~$X$, we deduce from L\'evy's theorem
that $B$ is $(\mathbb F,Q)$~Brownian motion with that quadratic
covariation. To see that $B$ is indeed a martingale, first note that
each component of $E^Q [ \pi(W_t) \mid\mathcal F_t ]$ is the
projection of a linear combination of martingales, hence itself a
martingale. Next, we make use of the following well-known result from
filtering theory (see~\cite{LiptserShiryaev1977}, Theorem~7.12): if
$\xi$ is a measurable process with $\int_0^t E^Q[|\xi_s|]\,ds<\infty$ for
all $t\ge0$, then
\[
E^Q \biggl[ \int_0^t
\xi_s \,ds \biggm|\mathcal F_t \biggr] - \int
_0^t E^Q [ \xi_s \mid
\mathcal F_s ]\,ds, \qquad t\ge0,
\]
is an $(\mathbb F,Q)$ martingale. Applying this to each component of
$L$ shows that it is a martingale. This completes the proof.
\end{pf}

We now have a description of the law of $X$ under~$P$ and under~$Q$. It
remains to show that these laws are locally equivalent, and this is
where condition~(\ref{eqlbm2}) is crucial. A priori, (\ref{eqlbm2})
only asserts boundedness on compact sets \emph{bounded away from $\{0\}
\times D$.} The following result shows that this can be strengthened
without imposing any additional assumptions. The proof uses the
Moore--Penrose inverse to decompose $Y_t$ into an observable component
and an independent component.

%
\begin{lemma} \label{Llbm2s}
Assume condition~(\ref{eqlbm2}) is satisfied. Then there is some
$\varepsilon>0$, and an open set $O\subset D$ containing $X_0$, such
that the function in~(\ref{eqlbm2}) is bounded on $(0,\varepsilon
]\times O$.
\end{lemma}

\begin{pf}
Define $G(y)=h(y)^{-1}\llvert  \pi(\nabla\ln h(y))\rrvert $, and let
$\pi^+$ be the Moore--Penrose inverse of the linear map $\pi$. Since $\pi
^+$ is invertible on $D$ (its inverse is~$\pi$), the function in~(\ref
{eqlbm2}) can be written
\[
E^P \bigl[ G(Y_t) \mid\pi(Y_t)=x \bigr] =
E^P \bigl[ G(Y_t) \mid U_t=\pi ^+(x) \bigr],
\]
where we set $U_t=\pi^+\pi(Y_t)$. Now decompose $Y_t$ as
\[
Y_t = \pi^+\pi(Y_t) + \bigl(\operatorname{Id} - \pi^+\pi
\bigr) (Y_t) = U_t + V_t
\]
($V_t$ is defined by this relation), and note that
\[
\pi^+\pi\bigl(\operatorname{Id} - \pi^+\pi\bigr) = \pi^+\pi- \pi^+\pi\pi ^+\pi=
\pi ^+\pi- \pi^+\pi= 0
\]
by basic properties of the Moore--Penrose inverse. Hence $Y_t=U_t +
V_t$ is the decomposition of $Y_t$ as a direct sum in $D\oplus D^\perp
$. In particular $U_t$ and $V_t$ are independent under $P$, so
\[
E^P \bigl[ G(Y_t) \mid U_t=\pi^+(x) \bigr] =
E^P \bigl[ G(u + V_t) \bigr]_{u=\pi^+(x)}.
\]
We now focus on bounding $E^P[ G(z + V_t)]$. The random variable $V_t$
concentrates on $D^\perp$ and is nondegenerate Normal there, so it has
a density with respect to Lebesgue measure on $D^\perp$ given by
\[
f_t(v) = \frac{1}{(2\pi t)^{m/2}|\det\Sigma|^{1/2}} \exp \biggl( -\frac
{1}{2t}
(v-V_0)^\top\Sigma^{-1} (v-V_0)
\biggr), \qquad v \in D^\perp.
\]
Here $m=q-p=\dim D^\perp$ and, by a slight abuse of notation, $\Sigma
^{-1}$ the inverse on~$D^\perp$ of the covariance operator of~$V_t$,
with $\det\Sigma$ being its determinant.

Now, let $\varepsilon>0$ be a number to be determined later. We let
$\mathcal B=\{u\in D\dvtx |u-U_0|<\varepsilon\}$ be the ball in $D$ of
radius $\varepsilon$ centered at $U_0$, and $\mathcal E$ be the
ellipsoid in $D^\perp$ given by
\[
\mathcal E = \biggl\{ v \in D^\perp\dvtx  \frac{1}{m}
(v-V_0)^\top\Sigma ^{-1} (v-V_0) <
\varepsilon \biggr\}.
\]
The following can be verified by direct differentiation:

\textit{Claim}: Fix $\alpha>0$ and $\beta>0$, and let $\psi
(t)=t^{-\alpha
/2}\exp(-t^{-1}\beta/2)$. Then $\psi$ is nondecreasing on the interval
$[0,\beta/\alpha]$.

The claim shows that whenever $v \notin\mathcal E$, $f_t(v)$ decreases
as $t$ decreases. This gives us the following bound for any $t\in
(0,\varepsilon]$:
\begin{eqnarray*}
E^P \bigl[ G(z+ V_t) \bigr] &=& \int_{\mathcal E}
G(u+v)f_t(v) \,dv + \int_{D^\perp\setminus\mathcal E} G(u+v)f_t(v)
\,dv
\\
&\le&\sup_{v\in\mathcal E} G(u+v) + \int_{D^\perp\setminus\mathcal E}
G(u+v)f_\varepsilon(v) \,dv
\\
&\le&\sup_{v\in\mathcal E} G(u+v) + E^P \bigl[
G(z+V_\varepsilon ) \bigr].
\end{eqnarray*}
Therefore,
\[
\sup_{(t,u)\in(0,\varepsilon]\times\mathcal B} E^P \bigl[ G(u+ V_t)
\bigr] \le\sup_{y\in\mathcal B\oplus\mathcal E}G(y) + \sup_{u\in
\mathcal
B}E^P
\bigl[ G(u+V_\varepsilon) \bigr].
\]
By smoothness of $h$ outside $E_0$ and the fact that $h(Y_0)=1$, it is
possible to choose $\varepsilon>0$ small enough that the set $\mathcal
B\oplus\mathcal E$, which is a neighborhood of $Y_0$, is bounded away
from $E_0$. With such an $\varepsilon$, the first term on the
right-hand side above is finite. The second term is also finite due to
the local boundedness assumption~(\ref{eqlbm2}). Setting $O=\pi
(\mathcal B)$, which is again open in $D$, gives the statement of the lemma.
\end{pf}

The same orthogonal decomposition of $Y_t$ as in the proof of
Lemma~\ref{Llbm2s} gives the following unsurprising result.

%
\begin{lemma} \label{Lexpexp}
Consider a nonnegative measurable function $G\dvtx \mathbb E\to\mathbb
R_+$. The equality
\[
E^P \bigl[ G(Y_t) \mid\mathcal F_t \bigr] =
E^P \bigl[ G(Y_t) \mid\pi(Y_t) = x
\bigr]_{x=X_t}
\]
holds $P$-a.s.~for all $t\ge0$.
\end{lemma}

\begin{pf}
With the notation from the proof of Lemma~\ref{Llbm2s} we get, $P$-a.s.,
\begin{eqnarray*}
E^P \bigl[ G(Y_t) \mid\mathcal F_t \bigr]
&=&E^P \bigl[ G(Y_t) \mid X_s\dvtx s\le t \bigr]
\\
&=&E^P\bigl[ G(U_t+V_t) \mid
U_s\dvtx s\le t\bigr]
\\
&=&E^P\bigl[ G(u+V_t)\bigr]_{u=U_t}
\\
&=&E^P\bigl[ G\bigl(\pi^+(x)+V_t\bigr)
\bigr]_{x=X_t}.
\end{eqnarray*}
By means of an analogous calculation, the right-hand side is also seen
to be equal to $E^P [ G(Y_t) \mid\pi(Y_s) = x  ]_{x=X_s}$.
\end{pf}

The following simple refinement of Bayes's rule is useful for dealing
with nonequivalent measures.

%
\begin{lemma} \label{Lbrref1}
Suppose $R_1\ll R_2$ are two probability measures with Radon--Nikodym
derivative $Z=\frac{dR_1}{dR_2}$, and let $X$ be a random variable in
$L^1(R_1)$. Let $\mathcal H$ be a sub-$\sigma$-field, and suppose
$A\in
\mathcal H$ satisfies $A\subset\{ E^{R_2}[Z \mid\mathcal H] > 0\}$.
Then \mbox{$E^{R_1} [ X \mid\mathcal H  ]$} is uniquely defined on
$A$ up to an $R_2$-nullset, and we have
\[
E^{R_2} [Z\mid\mathcal H ]E^{R_1} [X\mid\mathcal H ]
\mathbf1_A=E^{R_2} [Z X \mathbf1_A \mid\mathcal H
]
\]
$R_2$-a.s. (and hence $R_1$-a.s.).
\end{lemma}

\begin{pf}
To prove the first statement, let $Y$ and $Y'$ be two versions of
\mbox{$E^{R_1}[X\mid\mathcal H]$}. Then $R_1(Y\neq Y')=0$, and we get
\[
0 = R_1 \bigl( \bigl\{Y\neq Y'\bigr\}\cap A \bigr) =
E^{R_2} \bigl[ E^{R_2} [Z \mid\mathcal H ]
\mathbf1_{\{Y\neq Y'\}\cap A} \bigr].
\]
Since $E^{R_2}[Z\mid\mathcal H]>0$ on $A$, we get $R_2(\{Y\neq Y'\}
\cap
A)=0$, as desired. The second statement follows from the following
calculation, where $B\in\mathcal H$ is arbitrary:
\begin{eqnarray*}
E^{R_2} \bigl[ E^{R_2} [ Z \mid\mathcal H ] E^{R_1} [
X\mid \mathcal H ] \mathbf1_{A\cap B} \bigr] &=& E^{R_2} \bigl[ Z
E^{R_1} [ X\mid\mathcal H ] \mathbf 1_{A\cap
B} \bigr]
\\
&=& E^{R_1} [ X \mathbf1_{A\cap B} ]
\\
&=& E^{R_2} [Z X \mathbf1_{A\cap B} ].
\end{eqnarray*}\upqed
\end{pf}

The next lemma is the key to proving that the laws of $X$ under $P$ and
$Q$ are equivalent. It relies on the strengthening of condition~(\ref
{eqlbm2}) given in Lemma~\ref{Llbm2s}.

%
\begin{lemma} \label{LThF}
Assume that (\ref{eqlbm1}) and (\ref{eqlbm2}) are satisfied, and let
$\theta$ and $Q$ be as in Lemma~\ref{LXdec}. For each $t\geq0$, we have
\[
\int_0^t |\theta_s|^2
\,ds < \infty\qquad Q\mbox{-a.s.}
\]
\end{lemma}

\begin{pf}
We would like to rewrite $\theta_t$ using Lemma~\ref{Lbrref1}, so we
verify the assumptions of that lemma. To this end, define
\[
\sigma_0 = \inf \bigl\{ t\geq0\dvtx  Q(\tau>t\mid\mathcal
F_t) = 0 \bigr\} = \inf \bigl\{ t\geq0\dvtx  E^Q[M_t
\mid\mathcal F_t] = 0 \bigr\},
\]
where the equality follows from Lemma~\ref{Lme1}. Then $\tau\leq
\sigma
_0$, $Q$-a.s., so the expression for $\theta$ yields
\[
\theta_t {\mathbf1_{\{\sigma_0\leq t\}}} = E^Q \bigl[\pi\bigl(
\nabla\ln h(Y_t)\bigr) \mathbf1_{\{
\tau>t\}\cap\{\sigma_0\leq t\}}\mid\mathcal
F_t \bigr] = 0.
\]
Hence $\theta_t = \theta_t{\mathbf1_{\{\sigma_0>t\}}}$. Now, set
$H=\pi(\nabla\ln
h(Y_t)){\mathbf1_{\{\tau>t\}}}$. Then
\[
E^P \bigl[|H| \bigr]=E^Q \bigl[|M_tH| \bigr]=E^Q \bigl[|
\theta _t| \bigr],
\]
which is finite by Lemma~\ref{LXdec}. Since also $E^P[M_t\mid
\mathcal
F_t]>0$ on $\{\sigma_0>t\}$, we may apply Lemma~\ref{Lbrref1} with
$R_1=P$ and $R_2=Q$ to get, $Q$-a.s.,
\begin{eqnarray*}
\theta_t &=& E^Q \bigl[ \pi\bigl(\nabla\ln
h(Y_t)\bigr) {\mathbf1_{\{\tau>t\}}}\mid \mathcal F_t
\bigr] {\mathbf1_{\{\sigma_0>t\}}}
\\
&=& E^P \biggl[ \frac{1}{h(Y_t)} \pi\bigl(\nabla\ln
h(Y_t)\bigr) \biggm|\mathcal F_t \biggr] E^Q [
M_t \mid\mathcal F_t ].
\end{eqnarray*}
Now, since $E^Q [ M_t \mid\mathcal F_t ]$ is a finite, c\`
adl\`
ag process, it is pathwise bounded on each $[0,t]$ (with the bound
depending on $t$ and $\omega$ in a possibly nonpredictable way). It
thus suffices to prove that $\int_0^{t\wedge\sigma_0} |\xi_s|^2
\,ds<\infty$, $Q$-a.s., where $\xi_s = E^P[ h(Y_s)^{-1} \pi(\nabla\ln
h(Y_s)) \mid\mathcal F_s]$. By Lemma~\ref{Lbrref1} this conditional
expectation is uniquely defined $P$- and $Q$-a.s. on $\{s<\sigma_0\}$.
Therefore, by Lemma~\ref{Lexpexp}, the equality
\[
\xi_s = E^P \biggl[\frac{1}{h(Y_s)} \pi\bigl(\nabla
\ln h(Y_s)\bigr) \biggm|\pi(Y_s) = x
\biggr]_{x=X_s}
\]
holds $Q$-a.s.~on $\{s<\sigma_0\}$.\vadjust{\goodbreak}

Now, let $O\subset D$ and $\varepsilon>0$ be the objects obtained from
Lemma~\ref{Llbm2s}, and define
\[
\rho_\varepsilon=\inf\{0\le t\le\varepsilon\wedge\sigma_0\dvtx
X_t\notin O\}.
\]
Since $O$ is open and contains $X_0$, we have $\rho_\varepsilon>0$,
$Q$-a.s. (Note that $\sigma_0>0$ by right continuity of $E^Q[M_t\mid
\mathcal F_t]$.) The properties of $O$ and $\varepsilon$ imply that
$\xi
_s$ is bounded on $(0,\rho_\varepsilon)$. Furthermore, the local
boundedness condition~(\ref{eqlbm2}) implies that $\xi_s$ is pathwise
bounded on $[\rho_\varepsilon,t\wedge\sigma_0)$ (again with a random
bound). It follows that $\xi$ is square integrable on $(0,t\wedge
\sigma
_0)$, which is what we had to show. The proof of the lemma is now complete.
\end{pf}

\begin{pf*}{Proof of Theorem~\ref{Tmain11}}
We need to prove that $Q$ and $P$ are equivalent on each $\mathcal
F_t$. By Lemmas~\ref{LXdec} and~\ref{LThF}, we can define a strictly
positive $(\mathbb F,Q)$~local martingale $Z$ via
\[
Z_t = \exp \biggl( \int_0^t
\theta_s^\top \,dB_s - \frac{1}{2}\int
_0^t |\theta_s|^2 \,ds
\biggr), \qquad t\ge0.
\]
Consequently, since $\mathbb F$ is a standard system, we can find the
F\"ollmer measure associated with $Z$. To be precise, define stopping times
\[
\rho_n = \inf\{t\ge0\dvtx  Z_t \ge n\}, \qquad\rho= \lim
_{n\to\infty
} \rho_n.
\]
Then there is a unique probability $R_0$ on $\mathcal F_{\rho-}$ such
that $\frac{dQ}{dR_0} |_{\rho_n-}=\frac{1}{Z_{\rho_n}}$ for each
$n$. Girsanov's theorem and L\'evy's characterization of Brownian
motion then imply that the process
\[
X_{t\wedge\rho_n} - X_0 = B_{t\wedge\rho_n} - \int
_0^{t\wedge\rho_n} \theta_s \,ds, \qquad t\ge0,
\]
is Brownian motion (with some invertible volatility matrix) stopped at
$\rho_n$. Moreover, since $X$ generates the filtration $\mathbb F$,
$\rho_n$ only depends on the path of $X$. Therefore the law of
$(X_{t\wedge\rho_n}\dvtx t\ge0)$ under $R_0$ is the same as its law under
$P$. \mbox{Consequently,} since $\int_0^t\theta^2_s\,ds<\infty$ for all $t\ge0$,
$P$-a.s., so that $P(\rho=\infty)=1$, we also have $R_0(\rho=\infty
)=1$. It follows that $X-X_0$ (not stopped this time) is Brownian
motion under $R_0$, and we deduce that $R_0=P$ on each $\mathcal F_t$.
This leads to the domination relations
\[
P|_{\mathcal F_t} \ll Q|_{\mathcal F_t} \ll R_0 |_{\mathcal F_t} =
P|_{\mathcal F_t},
\]
which proves the theorem.
\end{pf*}

\section{Examples} \label{Sex}

In this section we discuss some examples where the conditions of
Theorem~\ref{Tmain11} can be verified explicitly. We also give one
recipe for how new examples can be constructed from old ones.\vadjust{\goodbreak}

%
\begin{example}[(The inverse Bessel process)]
Let $E=\mathbb R^3$, and suppose $Y_0=(1,0,0)$. Take $h(y)=|y|$. Then
$1/h$ is harmonic on $\mathbb R^3\setminus\{0\}$, and $N$ is the
reciprocal of a BES(3) process. In particular it is a strict local
martingale. To specify the smaller filtration, we let $\pi$ be a
projection onto the first coordinate of~$\mathbb R^3$. This puts us
exactly in the example analyzed by F\"ollmer and Protter \cite
{FollmerProtter2011}, mentioned in the \hyperref[SB]{Introduction}.

Let us verify conditions (\ref{eqlbm1}) and (\ref{eqlbm2}) of
Theorem~\ref{Tmain11}. First, note that $\nabla h(y) = y |y|^{-1}$,
so that
\[
E^P \biggl[ \frac{1}{h(Y_t)} \bigl|\nabla\ln h(Y_t)\bigr|
\biggr] = E^P \biggl[ \frac
{1}{h(Y_t)^2} \biggr] = E^P
\bigl[N_t^2 \bigr].
\]
The well-known fact that $t\mapsto E^P[N_t^2]$ is bounded (see
Chapter~1.10 in~\cite{ChungWilliams1990}) directly implies~(\ref
{eqlbm1}). To prove~(\ref{eqlbm2}), write
\begin{eqnarray*}
F(t,x) &=& E^P \biggl[ \frac{1}{h(Y_t)}\bigl| \pi\bigl(\nabla\ln
h(Y_t)\bigr)\bigr| \Bigm |\pi (Y_t)=x \biggr] =
E^P \biggl[ \frac{|Y^1_t|}{|Y_t|^3} \biggm|Y^1_t=x
\biggr]
\\
&=& E^P \biggl[ \frac{|x|}{[x^2 + (Y^2_t)^2 + (Y^3_t)^2]^{3/2}} \biggr],
\end{eqnarray*}
where the last equality follows from the independence of the components
of $Y$. By the scaling property of Brownian motion,
$F(t,x)=t^{-1}F(1,t^{-1/2}x)$. To prove local boundedness of $F$ on
$(0,\infty)\times\mathbb R$ it is therefore enough to show that
$x\mapsto F(1,x)$ is locally bounded on $\mathbb R$. Noting that the
random variable $Z=(Y^2_1)^2 + (Y^3_1)^2$ is $\chi^2_2$~distributed,
we obtain
\begin{eqnarray*}
F(1,x) &=& E^P \biggl[ \frac{|x|}{(x^2 + Z)^{3/2}} \biggr]
\\
&=& \frac{|x|}{2} \int_0^\infty
\bigl(x^2+z\bigr)^{-3/2} e^{-z/2} \,dz
\\
&\leq&\frac{|x|}{2} \int_0^\infty
\bigl(x^2+z\bigr)^{-3/2}\,dz = 1.
\end{eqnarray*}
We thus obtain~(\ref{eqlbm2}), as required.

To connect this example with the theory developed in the previous
sections, note that the set $D_0=\pi\circ h^{-1}(\{0\})$ is simply
equal to $\{0\} \subset\mathbb R$. Theorem~\ref{TdL} then tells us
that the process $\Lambda$ only increases on the set $\{t\dvtx Y^1_t=0\}$.
In view of Proposition~\ref{PNdec}, this explains the appearance of
the local time in the expression for $E^P[N_t\mid\mathcal F_t]$ found
by F\"ollmer and Protter; see (\ref{eqFP}) in the \hyperref[SB]{Introduction}.
\end{example}

%
\begin{example}[(The inverse Bessel process embedded in $\mathbb R^4$)]
\label{ex2}
We now consider what happens when the previous example is embedded in
$\mathbb R^4$. Thus, we set \mbox{$E=\mathbb R^4$}, and let $Y$ start from
$(1,0,0,0)$. The function $h$ is now given by
\[
h(y) = |\bar y|\qquad\mbox{where }\bar y = (y_1,y_2,y_3).
\]
In other words, $h(y)$ is the distance between $y$ and the $y_4$-axis.
Then $N_t=1/h(Y_t)$ is a again the reciprocal of a BES(3) process, and
again a strict local martingale. It is clear that $1/h$ is harmonic
outside the $y_4$-axis, $E_0=\{y\dvtx  \bar y=0\}$. We let $\pi$ be
given by the following matrix representation in the canonical basis on
$\mathbb R^4$:
\[
\pi(y) = Ay\qquad\mbox{where } A=\pmatrix{ 1 & 0 & 0 & \alpha_1
\cr
0 & 1 & 0 & \alpha_2
\cr
0 & 0 & 0 & 0
\cr
0 & 0 & 0 & 0}\qquad
\mbox{for some } \alpha_1,\alpha_2 \in\mathbb R\setminus
\{0\}.
\]
Note that $D=\pi(E)$ can be identified with $\mathbb R^2$. We proceed
to verify conditions (\ref{eqlbm1}) and (\ref{eqlbm2}). First, the
gradient of $h$ is given by
\[
\nabla h(y) = \biggl( \frac{\bar y}{|\bar y|}, 0 \biggr) \in\mathbb
R^4.
\]
Hence
\[
E^P \biggl[ \frac{1}{h(Y_t)} \bigl|\nabla\ln h(Y_t)\bigr|
\biggr] = E^P \bigl[ N_t^2 \bigr]
\]
and we get (\ref{eqlbm1}) as in the previous example. We continue
with~(\ref{eqlbm2}), and define
\[
F(t,x) =\pmatrix{ F_1(t,x)
\cr
F_2(t,x)}, \qquad
F_i(t,x)=E^P \biggl[ \frac{|\pi(\nabla\ln h(Y_t))_i|}{h(Y_t)} \biggm |\pi
(Y_t)=x \biggr].
\]
Using the definition of $h$, the expression for $\nabla h$, and the
definition of $\pi$, one gets
\[
F_i(t,x) = E^P \biggl[ \frac{|Y^i_t|}{|\widebar Y_t|^3} \biggm|\pi
(Y_t) = x \biggr], \qquad i=1,2.
\]
The Brownian scaling property again shows that
$F(t,x)=t^{-1}F(1,t^{-1/2}x)$, so just as in the previous example we
need only consider $F(1,x)$. Next,
%
%
\begin{equation}
\label{eqFi1} F_i(1,x) \le E^P \biggl[
\frac{|Y^i_1|}{[ (Y^i_1)^2 + (Y^3_1)^2 ]^{3/2}} \biggm|\pi(Y_1) = x \biggr].
\end{equation}
To continue, we need to know the distribution of $(Y^i_1,Y^3_1)$
conditionally on $\pi(Y_1)=x$, for $i=1,2$. This can, for instance, be
done using the formula for the conditional multivariate Normal, applied
to the multivariate Normal vector $(Y^i_1,Y^3_1,\pi(Y_1))$. The result
of this calculation is that $Y^i_1$ and $Y^3_1$ are conditionally
independent, with $Y^3_1$ having mean zero and unit variance, and
$Y^i_1$, $i=1,2$, satisfying
\begin{eqnarray*}
\mu_1 &=& E\bigl[Y^1_1\mid
\pi(Y_1)=x\bigr] = 1 + \frac{(\alpha_2^2+1)(x_1-1) -
\alpha_1\alpha_2x_2}{1+\alpha_1^2+\alpha_2^2},
\\
\mu_2 &=& E\bigl[Y^2_1\mid
\pi(Y_1)=x\bigr] = 1 + \frac{(\alpha_1^2+1)x_2 -
\alpha
_1\alpha_2(x_1-1)}{1+\alpha_1^2+\alpha_2^2},
\\
\sigma_i^2 &=& \operatorname{Var}\bigl[Y^i_1
\mid\pi(Y_1)=x\bigr] = \frac
{\alpha
_i^2}{1+\alpha_1^2 + \alpha_2^2}.
\end{eqnarray*}
Continuing from~(\ref{eqFi1}) and using that $\alpha_1$ and $\alpha_2$
are nonzero,
\[
F_i(1,x) \le\frac{1}{2\pi\sigma_i} \int_{-\infty}^\infty
\int_{-\infty}^\infty\frac{|u|}{(u^2 + v^2)^{3/2}} \exp \biggl( -
\frac
{(u-\mu
_i)^2}{2\sigma_i^2} - \frac{v^2}{2} \biggr) \,du \,dv.
\]
Now split the inner integral (with variable $u$) into two parts: the
first over $(-1,1)$ and the second over $\mathbb R\setminus(-1,1)$.
Starting with the first part, we get
\begin{eqnarray*}
&&\frac{1}{2\pi\sigma_i} \int_{-\infty}^\infty\int
_{-1}^1 \frac
{|u|}{(u^2 + v^2)^{3/2}} \exp \biggl( -
\frac{(u-\mu_i)^2}{2\sigma
_i^2} - \frac{v^2}{2} \biggr) \,du \,dv
\\
&&\qquad\le\frac{1}{2\pi\sigma_i} \int_{-\infty}^\infty\int
_{-1}^1 \frac
{|u|}{(u^2 + v^2)^{3/2}}\,du\, e^{ - v^2/2} \,dv
\\
&&\qquad = \frac{1}{\pi\sigma_i} \int_{-\infty}^\infty \bigl(
\sqrt{1 + v^2} - \sqrt{v^2} \bigr) e^{ - v^2/2} \,dv
\\
&&\qquad \le\sqrt{\frac{2}{\pi}}\frac{1}{\sigma_i},
\end{eqnarray*}
where the last line used the inequality $\sqrt{a^2+b^2}\le|a|+|b|$ and
then the fact that the Normal density integrates to one. We now
consider the integral over the complementary set $\mathbb R\setminus
(-1,1)$. Since $u^2\ge1$ there, we get
\begin{eqnarray*}
&&\frac{1}{2\pi\sigma_i}  \int_{-\infty}^\infty\int
_{\mathbb
R\setminus
(-1,1)} \frac{|u|}{(u^2 + v^2)^{3/2}} \exp \biggl( -\frac{(u-\mu
_i)^2}{2\sigma_i^2} -
\frac{v^2}{2} \biggr) \,du \,dv
\\
&&\qquad \le\frac{1}{2\pi\sigma_i}\int_{-\infty}^\infty\int
_{\mathbb
R\setminus(-1,1)}|u| \exp \biggl( -\frac{(u-\mu_i)^2}{2\sigma_i^2} -
\frac
{v^2}{2} \biggr) \,du \,dv
\\
&&\qquad \le E^P \bigl[ \bigl|Y^i_1\bigr| \mid
\pi(Y_1)=x \bigr].
\end{eqnarray*}
The right-hand side is the expectation of a \emph{folded Normal}
distribution, and its value is a smooth function of $\mu_i$; see \cite
{Leoneetal1961} or compute directly. Consequently it is a locally
bounded function of $x$, and this finally shows that (\ref{eqlbm2}) holds.

Finally, note that $D_0 = \pi(E_0) = \{(\lambda\alpha_1,\lambda
\alpha
_2)\dvtx \lambda\in\mathbb R\}$. This is a proper subspace in $D=\mathbb
R^2$, and in particular it is Lebesgue-null. We would therefore expect
that the semimartingale decomposition of the projection of $N$ onto
$\mathbb F$ in this case also has a singular component.
\end{example}

%
\begin{example}[(A counterexample)]
Consider again the situation in Example~\ref{ex2}, but this time set
$\alpha_1=\alpha_2=0$. Then $Y^4$ does not play any role at all, and
$\mathbb F$ is generated by $(Y^1,Y^2)$. In this case the equivalent
measure extension problem has no solution---indeed, this corresponds to
projecting the inverse Bessel process onto the filtration $\mathbb
F^{1,2}$ mentioned in the \hyperref[SB]{Introduction}, and according to F\"ollmer and
Protter's results (Theorem~5.2 in~\cite{FollmerProtter2011}) this
projection is again a local martingale. Corollary~\ref{Cns} then shows
that no solution to the equivalent measure extension problem can be
found. Condition~(\ref{eqlbm2}) can therefore not be satisfied, and
this can indeed be verified directly: with $F_i(t,x)$ as in
Example~\ref{ex2}, we have
\begin{eqnarray*}
\bigl|F_i(1,x)\bigr| &=& E^P \biggl[ \frac{|x_i|}{[x_1^2 + x_2^2 + (Y^3_1)^2]^{3/2}} \biggr]
\\
&\ge&\frac{1}{\sqrt{2\pi e}} \int_{-1}^1
\frac
{|x_i|}{(x_1^2+x_2^2+u^2)^{3/2}} \,du
\\
&=& \sqrt{\frac{2}{\pi e}} \frac{|x_i|}{(x_1^2+x_2^2)\sqrt{1+x_1^2+x_2^2}}.
\end{eqnarray*}
The right-hand side is unbounded near the origin.
\end{example}

%
\begin{example}[(Building new examples from old)]
Suppose we have functions $h_1,\ldots,h_m$ such that for each $i$,
$1/h_i$ is harmonic outside $h_i^{-1}(\{0\})$. We define the set
\[
E_0 = \bigcup_{i=1}^m
h_i^{-1}\bigl(\{0\}\bigr)
\]
as the collection of points where some $h_i$ vanishes. We may then
define $h$ by
\[
\frac{1}{h} = \frac{1}{h_1} + \cdots+ \frac{1}{h_m} \qquad\mbox
{on } E\setminus E_0
\]
and extend it continuously to all of $E$ by setting $h(y)=0$, $y\in
E_0$. We have the following result.

%
\begin{lemma}
Consider $h$ and $E_0$ as above. The function $1/h$ is harmonic outside
$E_0$, and we have
\[
\frac{1}{h}\nabla\ln h = \frac{1}{h_1}\nabla\ln h_1 +
\cdots+ \frac
{1}{h_m}\nabla\ln h_m.
\]
\end{lemma}

\begin{pf}
By linearity of the Laplacian it is clear that $\frac{1}{h}$ is
harmonic. The second statement follows from the following elementary
calculation:
\begin{eqnarray*}
\nabla h &=& \nabla \biggl[ \biggl( \frac{1}{h_1} + \cdots+
\frac{1}{h_m} \biggr)^{-1} \biggr]
\\
&=& - \biggl( \frac{1}{h_1} + \cdots+ \frac{1}{h_m}
\biggr)^{-2} \biggl( \nabla \biggl(\frac{1}{h_1} \biggr) + \cdots+
\nabla \biggl(\frac
{1}{h_m} \biggr) \biggr)
\\
&=& h^2 \biggl( \frac{1}{h_1}\nabla\ln h_1 + \cdots+
\frac
{1}{h_m}\nabla\ln h_m \biggr).
\end{eqnarray*}\upqed
\end{pf}

It follows directly from this lemma that if each $h_i$ satisfies (\ref
{eqlbm1}) and~(\ref{eqlbm2}), then the same will be true for $h$. A
simple application of this result is that any process $N$ of the form
\[
N_t = \frac{1}{|Y_t - y^{(1)}|} + \cdots+ \frac{1}{|Y_t - y^{(m)}|},
\]
where $y^{(1)},\ldots,y^{(m)} \in\mathbb R^3$ are fixed and different
from $Y_0$, induces a F\"ollmer measure that can be extended to an
equivalent measure on the subfiltration generated by $Y^1$.
\end{example}

\section{Applications in finance} \label{SApp}
We end with a brief discussion of some consequences for financial
modeling and arbitrage. The discussion will be kept on an informal
level, and we defer the development and analysis of concrete models to
future research. The notation from Sections~\ref{SB} and~\ref{SFM}
will be used freely. The first observation, which has been made
in~\cite
{FollmerProtter2011} and~\cite{JarrowProtter2013}, is that market
participants with limited information may perceive arbitrage
opportunities even if there are none. This interpretation arises when
$N$ is a price process, and less informed investors only see its
optional projection.

An alternative situation is the following. Consider a well-informed
fund manager with filtration~$\mathbb G$ who trades on behalf of less
informed investors with filtration~$\mathbb F$, in exchange for a fee.
Such arrangements are common, and arise \mbox{because} the fund manager has
superior information, and/or because he has cheaper (lower transactions
costs) access to the market. Suppose further that the measures~$P$
and~$Q$ represent competing beliefs regarding the future evolution of
the world, and suppose~$M$ is the value process of the fund manager's
investment strategy, where~$M$ reaching zero corresponds to bankruptcy.
If the beliefs~$Q$ (under which~$M$ may in fact hit zero) are
correct,~$M$ is a very risky investment. In contrast, under~$P$
bankruptcy happens with zero probability. The key point is that less
informed investors who estimate $M$ via its optional projection will
always obtain a strictly positive estimate, even if their beliefs are
correct and given by~$Q$ (where~$Q$ solves the equivalent measure
extension problem). In effect, the fund manager can run risky
strategies which, \emph{conditionally on no bankruptcy}, achieve
superior returns, while convincing investors that bankruptcy is
impossible. He can thus charge excessive fees, which allows him to
achieve arbitrage profits (for himself) by exploiting the fact that
investors are ill-informed.

Any model where effects of this type occur will necessarily include
components relating to the contractual relationship between investors
and fund manager, the investment horizon, what happens if $M$ does, in
fact, reach zero, and so forth. While such domain specific issues fall
outside the scope of the present paper, they are the subject of ongoing
research.

\section*{Acknowledgments}
The author would like to thank Robert Jarrow, Philip Protter and Johannes Ruf for
several very stimulating discussions. Thanks are also due to an anonymous referee, whose insightful comments led to several improvements of the paper.


%

\printaddresses


\begin{thebibliography}{28}
\bibitem{CarrFisherRuf2012}
%
\begin{barticle}[mr]
\bauthor{\bsnm{Carr},~\bfnm{Peter}\binits{P.}},
\bauthor{\bsnm{Fisher},~\bfnm{Travis}\binits{T.}} \AND
\bauthor{\bsnm{Ruf},~\bfnm{Johannes}\binits{J.}}
(\byear{2014}).
\btitle{On the hedging of options on exploding exchange rates}.
\bjournal{Finance Stoch.}
\bvolume{18}
\bpages{115--144}.
\bid{doi={10.1007/s00780-013-0218-3}, issn={0949-2984}, mr={3146489}}
\bptnote{check year}%
\end{barticle}
%
\bptok{imsref}%
\endbibitem

\bibitem{Cetinetal2004}
%
\begin{barticle}[mr]
\bauthor{\bsnm{{\c{C}}etin},~\bfnm{Umut}\binits{U.}},
\bauthor{\bsnm{Jarrow},~\bfnm{Robert}\binits{R.}},
\bauthor{\bsnm{Protter},~\bfnm{Philip}\binits{P.}} \AND
\bauthor{\bsnm{Yildirim},~\bfnm{Yildiray}\binits{Y.}}
(\byear{2004}).
\btitle{Modeling credit risk with partial information}.
\bjournal{Ann. Appl. Probab.}
\bvolume{14}
\bpages{1167--1178}.
\bid{doi={10.1214/105051604000000251}, issn={1050-5164}, mr={2071419}}
\end{barticle}
%
\bptok{imsref}%
\endbibitem

\bibitem{ChungWilliams1990}
%
\begin{bbook}[mr]
\bauthor{\bsnm{Chung},~\bfnm{K.~L.}\binits{K.~L.}} \AND
\bauthor{\bsnm{Williams},~\bfnm{R.~J.}\binits{R.~J.}}
(\byear{1990}).
\btitle{Introduction to Stochastic Integration},
\bedition{2nd} ed.
\bpublisher{Birkh\"auser},
\blocation{Boston, MA}.
\bid{doi={10.1007/978-1-4612-4480-6}, mr={1102676}}
\end{bbook}
%
\bptok{imsref}%
\endbibitem

\bibitem{DelbaenSchachermayer1995}
%
\begin{barticle}[mr]
\bauthor{\bsnm{Delbaen},~\bfnm{F.}\binits{F.}} \AND
\bauthor{\bsnm{Schachermayer},~\bfnm{W.}\binits{W.}}
(\byear{1995}).
\btitle{Arbitrage possibilities in {B}essel processes and their
relations to local martingales}.
\bjournal{Probab. Theory Related Fields}
\bvolume{102}
\bpages{357--366}.
\bid{doi={10.1007/BF01192466}, issn={0178-8051}, mr={1339738}}
\end{barticle}
%
\bptok{imsref}%
\endbibitem

\bibitem{DellacherieMeyer1978}
%
\begin{bbook}[auto:STB|2013/12/09|07:59:19]
\bauthor{\bsnm{Delleacherie},~\bfnm{C.}\binits{C.}} \AND
\bauthor{\bsnm{Meyer},~\bfnm{P.~A.}\binits{P.~A.}}
(\byear{1978}).
\btitle{Probabilities and Potential}.
\bpublisher{North-Holland},
\blocation{Amsterdam}.
\end{bbook}
%
\bptok{imsref}%
\endbibitem

\bibitem{DellacherieMeyer1982}
%
\begin{bbook}[mr]
\bauthor{\bsnm{Dellacherie},~\bfnm{Claude}\binits{C.}} \AND
\bauthor{\bsnm{Meyer},~\bfnm{Paul-Andr{\'e}}\binits{P.-A.}}
(\byear{1982}).
\btitle{Probabilities and Potential. {B}: Theory of Martingales}.
\bpublisher{North-Holland},
\blocation{Amsterdam}.
\bid{mr={0745449}}
\end{bbook}
%
\bptok{imsref}%
\endbibitem

\bibitem{Doob1957}
%
\begin{barticle}[mr]
\bauthor{\bsnm{Doob},~\bfnm{J.~L.}\binits{J.~L.}}
(\byear{1957}).
\btitle{Conditional {B}rownian motion and the boundary limits of
harmonic functions}.
\bjournal{Bull. Soc. Math. France}
\bvolume{85}
\bpages{431--458}.
\bid{issn={0037-9484}, mr={0109961}}
\end{barticle}
%
\bptok{imsref}%
\endbibitem

\bibitem{Fernholz2010fk}
%
\begin{barticle}[mr]
\bauthor{\bsnm{Fernholz},~\bfnm{Daniel}\binits{D.}} \AND
\bauthor{\bsnm{Karatzas},~\bfnm{Ioannis}\binits{I.}}
(\byear{2010}).
\btitle{On optimal arbitrage}.
\bjournal{Ann. Appl. Probab.}
\bvolume{20}
\bpages{1179--1204}.
\bid{doi={10.1214/09-AAP642}, issn={1050-5164}, mr={2676936}}
\end{barticle}
%
\bptok{imsref}%
\endbibitem

\bibitem{Follmer1972}
%
\begin{barticle}[mr]
\bauthor{\bsnm{F{\"o}llmer},~\bfnm{Hans}\binits{H.}}
(\byear{1972}).
\btitle{The exit measure of a supermartingale}.
\bjournal{Z. Wahrsch. Verw. Gebiete}
\bvolume{21}
\bpages{154--166}.
\bid{mr={0309184}}
\end{barticle}
%
\bptok{imsref}%
\endbibitem

\bibitem{FollmerProtter2011}
%
\begin{barticle}[mr]
\bauthor{\bsnm{F{\"o}llmer},~\bfnm{Hans}\binits{H.}} \AND
\bauthor{\bsnm{Protter},~\bfnm{Philip}\binits{P.}}
(\byear{2011}).
\btitle{Local martingales and filtration shrinkage}.
\bjournal{ESAIM Probab. Stat.}
\bvolume{15}
\bpages{S25--S38}.
\bid{doi={10.1051/ps/2010023}, issn={1292-8100}, mr={2817343}}
\end{barticle}
%
\bptok{imsref}%
\endbibitem

\bibitem{GuoZeng2008}
%
\begin{barticle}[mr]
\bauthor{\bsnm{Guo},~\bfnm{Xin}\binits{X.}} \AND
\bauthor{\bsnm{Zeng},~\bfnm{Yan}\binits{Y.}}
(\byear{2008}).
\btitle{Intensity process and compensator: A new filtration expansion
approach and the {J}eulin--{Y}or theorem}.
\bjournal{Ann. Appl. Probab.}
\bvolume{18}
\bpages{120--142}.
\bid{doi={10.1214/07-AAP447}, issn={1050-5164}, mr={2380894}}
\end{barticle}
%
\bptok{imsref}%
\endbibitem

\bibitem{HestonLoewensteinWillard2007}
%
\begin{barticle}[auto:STB|2013/12/09|07:59:19]
\bauthor{\bsnm{Heston},~\bfnm{S.}\binits{S.}},
\bauthor{\bsnm{Loewenstein},~\bfnm{M.}\binits{M.}} \AND
\bauthor{\bsnm{Willard},~\bfnm{G.}\binits{G.}}
(\byear{2007}).
\btitle{Options and bubbles}.
\bjournal{Rev. Financ. Stud.}
\bvolume{20}
\bpages{359--390}.
\end{barticle}
%
\bptok{imsref}%
\endbibitem

\bibitem{Hugonnier2012}
%
\begin{barticle}[mr]
\bauthor{\bsnm{Hugonnier},~\bfnm{Julien}\binits{J.}}
(\byear{2012}).
\btitle{Rational asset pricing bubbles and portfolio constraints}.
\bjournal{J. Econom. Theory}
\bvolume{147}
\bpages{2260--2302}.
\bid{doi={10.1016/j.jet.2012.05.003}, issn={0022-0531}, mr={2996647}}
\end{barticle}
%
\bptok{imsref}%
\endbibitem

\bibitem{Jacod2003fk}
%
\begin{bbook}[mr]
\bauthor{\bsnm{Jacod},~\bfnm{Jean}\binits{J.}} \AND
\bauthor{\bsnm{Shiryaev},~\bfnm{Albert~N.}\binits{A.~N.}}
(\byear{2003}).
\btitle{Limit Theorems for Stochastic Processes},
\bedition{2nd} ed.
\bseries{Grundlehren der Mathematischen Wissenschaften [Fundamental
Principles of Mathematical Sciences]}
\bvolume{288}.
\bpublisher{Springer},
\blocation{Berlin}.
\bid{mr={1943877}}
\end{bbook}
%
\bptok{imsref}%
\endbibitem

\bibitem{JarrowProtter2013}
%
\begin{barticle}[mr]
\bauthor{\bsnm{Jarrow},~\bfnm{Robert}\binits{R.}} \AND
\bauthor{\bsnm{Protter},~\bfnm{Philip}\binits{P.}}
(\byear{2013}).
\btitle{Positive alphas, abnormal performance, and illusory arbitrage}.
\bjournal{Math. Finance}
\bvolume{23}
\bpages{39--56}.
\bid{doi={10.1111/j.1467-9965.2011.00489.x}, issn={0960-1627}, mr={3015233}}
\end{barticle}
%
\bptok{imsref}%
\endbibitem

\bibitem{Jarrow2007uq}
%
\begin{barticle}[mr]
\bauthor{\bsnm{Jarrow},~\bfnm{Robert~A.}\binits{R.~A.}},
\bauthor{\bsnm{Protter},~\bfnm{Philip}\binits{P.}} \AND
\bauthor{\bsnm{Sezer},~\bfnm{A.~Deniz}\binits{A.~D.}}
(\byear{2007}).
\btitle{Information reduction via level crossings in a credit risk model}.
\bjournal{Finance Stoch.}
\bvolume{11}
\bpages{195--212}.
\bid{doi={10.1007/s00780-006-0033-1}, issn={0949-2984}, mr={2295828}}
\end{barticle}
%
\bptok{imsref}%
\endbibitem

\bibitem{JarrowProtterShimbo2010}
%
\begin{barticle}[mr]
\bauthor{\bsnm{Jarrow},~\bfnm{Robert~A.}\binits{R.~A.}},
\bauthor{\bsnm{Protter},~\bfnm{Philip}\binits{P.}} \AND
\bauthor{\bsnm{Shimbo},~\bfnm{Kazuhiro}\binits{K.}}
(\byear{2010}).
\btitle{Asset price bubbles in incomplete markets}.
\bjournal{Math. Finance}
\bvolume{20}
\bpages{145--185}.
\bid{doi={10.1111/j.1467-9965.2010.00394.x}, issn={0960-1627}, mr={2650245}}
\end{barticle}
%
\bptok{imsref}%
\endbibitem

\bibitem{JeulinYor1978}
%
\begin{bincollection}[mr]
\bauthor{\bsnm{Jeulin},~\bfnm{T.}\binits{T.}} \AND
\bauthor{\bsnm{Yor},~\bfnm{M.}\binits{M.}}
(\byear{1978}).
\btitle{Grossissement d'une filtration et semi-martingales: Formules
explicites}.
In \bbooktitle{S\'eminaire de {P}robabilit\'es, {XII} ({U}niv.
{S}trasbourg, {S}trasbourg, 1976/1977)}.
\bseries{Lecture Notes in Math.}
\bvolume{649}
\bpages{78--97}.
\bpublisher{Springer},
\blocation{Berlin}.
\bid{mr={0519998}}
\end{bincollection}
%
\bptok{imsref}%
\endbibitem

\bibitem{Kardarasetal2011}
%
\begin{bmisc}[auto:STB|2013/12/09|07:59:19]
\bauthor{\bsnm{Kardaras},~\bfnm{C.}\binits{C.}},
\bauthor{\bsnm{Kreher},~\bfnm{D.}\binits{D.}} \AND
\bauthor{\bsnm{Nikeghbali},~\bfnm{A.}\binits{A.}}
(\byear{2011}).
\bhowpublished{Strict local martingales and bubbles. Available at
\arxivurl{arXiv:1108.4177}.}
\end{bmisc}
%
\bptok{imsref}%
\endbibitem

\bibitem{Larsson2013}
%
\begin{bmisc}[auto:STB|2013/12/09|07:59:19]
\bauthor{\bsnm{Larsson},~\bfnm{M.}\binits{M.}}
(\byear{2013}).
\bhowpublished{Nonequivalent beliefs and subjective equilibrium bubbles.
Available at \arxivurl{arXiv:1306.5082}.}
\end{bmisc}
%
\bptok{imsref}%
\endbibitem

\bibitem{Leoneetal1961}
%
\begin{barticle}[mr]
\bauthor{\bsnm{Leone},~\bfnm{F.~C.}\binits{F.~C.}},
\bauthor{\bsnm{Nelson},~\bfnm{L.~S.}\binits{L.~S.}} \AND
\bauthor{\bsnm{Nottingham},~\bfnm{R.~B.}\binits{R.~B.}}
(\byear{1961}).
\btitle{The folded normal distribution}.
\bjournal{Technometrics}
\bvolume{3}
\bpages{543--550}.
\bid{issn={0040-1706}, mr={0130737}}
\end{barticle}
%
\bptok{imsref}%
\endbibitem

\bibitem{LiptserShiryaev1977}
%
\begin{bbook}[auto:STB|2013/12/09|07:59:19]
\bauthor{\bsnm{Liptser},~\bfnm{R.~S.}\binits{R.~S.}} \AND
\bauthor{\bsnm{Shiryaev},~\bfnm{A.}\binits{A.}}
(\byear{1977}).
\btitle{Statistics of Random Processes}.
\bpublisher{Springer},
\blocation{Berlin}.
\end{bbook}
%
\bptok{imsref}%
\endbibitem

\bibitem{Meyer1972}
%
\begin{bincollection}[mr]
\bauthor{\bsnm{Meyer},~\bfnm{P.~A.}\binits{P.~A.}}
(\byear{1972}).
\btitle{La mesure de {H}. {F}\"ollmer en th\'eorie des surmartingales}.
In \bbooktitle{S\'eminaire de {P}robabilit\'es, {VI}
({U}niv. {S}trasbourg, Ann\'ee Universitaire 1970--1971; {J}ourn\'ees
{P}robabilistes de {S}trasbourg, 1971)}
\bpages{118--129}.
\bpublisher{Springer},
\blocation{Berlin}.
\bid{mr={0368131}}
\end{bincollection}
%
\bptok{imsref}%
\endbibitem

\bibitem{PalProtter2010}
%
\begin{barticle}[mr]
\bauthor{\bsnm{Pal},~\bfnm{Soumik}\binits{S.}} \AND
\bauthor{\bsnm{Protter},~\bfnm{Philip}\binits{P.}}
(\byear{2010}).
\btitle{Analysis of continuous strict local martingales via \mbox
{$h$-}transforms}.
\bjournal{Stochastic Process. Appl.}
\bvolume{120}
\bpages{1424--1443}.
\bid{doi={10.1016/j.spa.2010.04.004}, issn={0304-4149}, mr={2653260}}
\end{barticle}
%
\bptok{imsref}%
\endbibitem

\bibitem{Parthasarathy1967}
%
\begin{bbook}[mr]
\bauthor{\bsnm{Parthasarathy},~\bfnm{K.~R.}\binits{K.~R.}}
(\byear{1967}).
\btitle{Probability Measures on Metric Spaces}.
\bpublisher{Academic Press},
\blocation{New York}.
\bid{mr={0226684}}
\end{bbook}
%
\bptok{imsref}%
\endbibitem

\bibitem{PlatenHeath2009}
%
\begin{bbook}[mr]
\bauthor{\bsnm{Platen},~\bfnm{Eckhard}\binits{E.}} \AND
\bauthor{\bsnm{Heath},~\bfnm{David}\binits{D.}}
(\byear{2006}).
\btitle{A Benchmark Approach to Quantitative Finance}.
\bseries{Springer Finance}.
\bpublisher{Springer},
\blocation{Berlin}.
\bid{doi={10.1007/978-3-540-47856-0}, mr={2267213}}
\bptnote{check year}%
\end{bbook}
%
\bptok{imsref}%
\endbibitem

\bibitem{Protter2005}
%
\begin{bbook}[mr]
\bauthor{\bsnm{Protter},~\bfnm{Philip~E.}\binits{P.~E.}}
(\byear{2005}).
\btitle{Stochastic Integration and Differential Equations},
\bedition{2nd} ed.
\bpublisher{Springer},
\blocation{Berlin}.
\bid{mr={2273672}}
\end{bbook}
%
\bptok{imsref}%
\endbibitem

\bibitem{Sezer2007}
%
\begin{barticle}[mr]
\bauthor{\bsnm{Sezer},~\bfnm{A.~Deniz}\binits{A.~D.}}
(\byear{2007}).
\btitle{Filtration shrinkage by level-crossings of a diffusion}.
\bjournal{Ann. Probab.}
\bvolume{35}
\bpages{739--757}.
\bid{doi={10.1214/009117906000000683}, issn={0091-1798}, mr={2308595}}
\end{barticle}
%
\bptok{imsref}%
\endbibitem

\end{thebibliography}
\end{document}